\documentclass[12pt]{amsart}
\usepackage{amsmath}

\usepackage{mathrsfs}
\usepackage{cases}
\usepackage{txfonts}
\usepackage{amsfonts}
\usepackage{mathrsfs}
\usepackage{amssymb}
\usepackage{amsmath,amssymb}
\usepackage{color}

\newtheorem{theorem}{Theorem}[section]
\newtheorem{thm}{Theorem}
\newtheorem{lemma}[theorem]{Lemma}
\newtheorem{proposition}[theorem]{Proposition}

\theoremstyle{definition}
\newtheorem{definition}[theorem]{Definition}

\newtheorem{problem}[theorem]{Problem}
\theoremstyle{remark}
\newtheorem{remark}[theorem]{Remark}
\numberwithin{equation} {section} %\linespread{1.1}
\begin{document}

\title[Fefferman-Stein decomposition and micro-local quantities]
 {Fefferman-Stein decomposition for $Q$-spaces  and micro-local quantities}

\thanks{The first author is supported by NSFC No. 11271209, 11571261.}

\thanks{The second author is supported by Research Grant of University of Macau
UL017/08-Y3/MAT/QT01/FST, Macao Science and Technology Fund FDCT/056/2010/A3.}
\thanks{Pengtao Li is the corresponding author and is  supported by NSFC No.11171203,  11201280;
New Teacher's Fund for Doctor Stations, Ministry of Education, China, No.20114402120003.}
%?



\maketitle

\centerline{\bf Qixiang Yang, } \centerline{{ the corresponding
author,}} \centerline{School of Mathematics and Statics, Wuhan
University, } \centerline{Wuhan, 430072, China.} \centerline{E-mail:
qxyang@whu.edu.cn}

\centerline{\bf Tao Qian,}
\centerline{ Department of Mathematics,
 University of Macau, }
  \centerline
 { Macao, China SAR.}
\centerline{ E-mail: fsttq@umac.mo}

\centerline{\bf  Pengtao Li,}
\centerline{{ the corresponding
author,}}
\centerline{ Department of Mathematics,
  Shantou University,}
  \centerline
 {Shantou, 515063, China.}
\centerline{ E-mail: ptli@stu.edu.cn}

\dedicatory{}
\date{}

\newpage
\centerline{\large \bf Fefferman-Stein decomposition for $Q$-spaces  and }
\centerline{\large \bf micro-local quantities}

\vspace{0.5cm}

\centerline{Qixiang Yang, Tao Qian and Pengtao Li }

\begin{abstract}
In this paper,  we consider the Fefferman-Stein decomposition of $Q_{\alpha}(\mathbb{R}^{n})$
and give an affirmative answer to an open problem posed by
M. Essen, S. Janson, L. Peng and J. Xiao in 2000.
One of our main methods is to study the structure of the predual space of
$Q_{\alpha}(\mathbb{R}^{n})$ by the micro-local quantities.
This result indicates that the norm of the predual space of
$Q_{\alpha}(\mathbb{R}^{n})$ depends on the micro-local structure
in a self-correlation way.

\end{abstract}
\tableofcontents
\pagenumbering{arabic}

\section{Introduction}
In this paper, we give a wavelet characterization of the predual of
$Q$-space $Q_{\alpha}(\mathbb{R}^{n})$ without using a family of Borel
measures. By this result, we obtain a Fefferman-Stein type
decomposition of $Q_{\alpha}(\mathbb{R}^{n})$.

Let $R_0$ be the unit operator and  $R_{i}, i=1,\cdots,n,$ be the Riesz transforms, respectively.
In 1972, in the celebrated paper \cite{FS}, C. Fefferman and E. M.
Stein proved the following result.
\begin{thm}\label{th1}
{\rm (\cite{U}, Theorem B)} If $f(x)\in BMO(\mathbb{R}^{n})$, then
there exist $g_{0}(x),\cdots$, $g_{n}(x)\in L^{\infty}(\mathbb{R}^{n})$
such that, modulo constants,  $f=\sum\limits_{j=0}^{n}R_{i}g_{i}$
and $\sum\limits^{n}_{j=0}\|g_{j}\|_{L^{\infty}}\leq C\|f\|_{BMO}$.
\end{thm}

The importance of Fefferman-Stein decomposition exists in two aspects.\\
On one hand, there is a close relation between the
$\bar{\partial}-$equation and Fefferman-Stein decomposition. On the
other hand, this decomposition helps us understand better the structure of
$BMO$ space and the distance between $L^{\infty}(\mathbb{R}^{n})$ and $BMO(\mathbb{R}^{n})$.
Due to the above two points, Fefferman-Stein decomposition of BMO
space has been studied extensively by many mathematicians since
1970s. We refer the readers to Jones \cite{J1, J2} and
Uchiyama \cite{U} for further information.
In latest decades, Fefferman-Stein decomposition is also extended to other function spaces,
for example, $BLO$, $C^0$ and $VMO$. We refer the reader to \cite{CR}, \cite{S} and the reference therein.

As an analogy of $BMO$, $Q$-spaces own the similar structure and many
common properties. It is natural to seek a Fefferman-Stien type
decomposition of $Q$-spaces. For the $Q$-spaces on unit disk, Nicolau-Xiao \cite{NX} obtained a
decomposition of $Q_{p}(\partial\mathbb{D})$ similar to Fefferman
and Stein's result of $BMO(\partial\mathbb{D})$ (see \cite{NX},
Theorem 1.2). On Euclidean space $\mathbb{R}^{n}$, Essen-Janson-Peng-Xiao \cite{EJPX}
introduced  $Q_{\alpha}(\mathbb{R}^{n})$ as a generalization of $Q_{p}(\partial\mathbb{D})$.
For $\alpha\in(-\infty,\infty)$,
$Q_{\alpha}(\mathbb{R}^{n})$ are defined as the spaces of all measurable
functions with
\begin{equation}\label{eq:111}\sup\limits_{I} |I|^{\frac{2\alpha}{n}-1}
\int_{I}\int_{I}\frac{|f(x)-f(y)|^{2}}{|x-y|^{n+2\alpha}}dxdy<\infty,
\end{equation}
where the supremum is taken over all cubes $I$ with
the edges parallel to the coordinate. They
studied  $Q_{\alpha}(\mathbb{R}^{n})$ systemically and
list Fefferman-Stein decomposition of $Q_{\alpha}(\mathbb{R}^{n})$
as one of the open problems.
\begin{problem}\label{p1}
(\cite{EJPX}, Problem 8.3) For $n\geq 2$ and  $\alpha\in(0,1)$.
Give a Fefferman-Stein type decomposition of
$Q_{\alpha}(\mathbb{R}^{n})$.
\end{problem}
In this paper, we will give an affirmative answer to this open
problem. Generally speaking, there are two methods to obtain the
Fefferman-Stein decomposition of $BMO(\mathbb{R}^{n})$. In \cite{FS}, C. Fefferman
and E. M. Stein split $BMO$ functions by an extension theorem, i.e.,
the Hahn-Banach theorem in functional analysis. In \cite{U}, A.
Uchiyama gave a constructive proof of Theorem \ref{th1}. In this
paper, using wavelets, we study the micro-local structure of
$P^{\alpha}(\mathbb{R}^{n})$, the predual of
$Q_{\alpha}(\mathbb{R}^{n})$. By this way, we
obtain a Fefferman-Stein type decomposition of
$Q_{\alpha}(\mathbb{R}^{n})$.

For the Fefferman-Stein decomposition of
$Q_{\alpha}(\mathbb{R}^{n})$, the difficulties exist in two aspects.

(1) For a function $f$ in $P^{\alpha}(\mathbb{R}^{n}), 0<\alpha<\frac{n}{2}$, the higher
frequency party and the lower frequency party make different
contributions to the norm $\|f\|_{P^{\alpha}}$. That is to say,
$P^{\alpha}(\mathbb{R}^{n})$ have special micro-local structure.

(2) For any function $f(x)$, Riesz transforms may cause  a
perturbation on all the range of its frequencies. To obtain Fefferman-Stein
type decomposition, we need to control the range of the
perturbation.

To overcome the above two difficulties, on one hand,
we analyze the micro-local structure of functions in
$P^{\alpha}(\mathbb{R}^{n})$.
Such micro-local structure could help
us get a wavelet characterization of $P^{\alpha}(\mathbb{R}^{n})$
without involving a group of Borel measures.
On the other hand, we use classical Meyer wavelets
to control the range of the perturbation. See also Remark \ref{open}.

In Section 2, we will give the definition of wavelet basis
$\{\Phi^{\epsilon}_{j,k}(x)\}_{(\epsilon, j, k)\in\Lambda_n}$.
It is well-known that a function $g(x)$ can be written as a sum
$$g(x)=\sum\limits_{j} g_{j}(x), \text{ where }
g_{j}(x)=: Q_{j}g(x)= \sum\limits_{\epsilon, k}
g^{\epsilon}_{j,k}\Phi^{\epsilon}_{j,k}(x).$$
Let $g(x)$ be a
function in Besov spaces or Triebel-Lizorkin spaces. Roughly speaking, the norms of
$g(x)$ can be determined by the $l^{p}(L^q)$-norm or $L^p(l^{q})$-norm of $\{g_j(x)\}$, respectively. See \cite{Me} and \cite{Tr}. For
$g(x) \in P^{\alpha}(\mathbb{R}^{n}), 0<\alpha<\frac{n}{2}$, the situation becomes
complicated and we can not use the above ideas. In Section 3, we
introduce the micro-local quantities with levels to study the
structure of functions in $P^{\alpha}(\mathbb{R}^{n})$.

Let $g(x)=\sum\limits_{\epsilon,j,k}g^{\epsilon}_{j,k}\Phi^{\epsilon}_{j,k}(x)
\in P^{\alpha}(\mathbb{R}^{n})$. For any dyadic cube $Q$, we take the function
$$g_Q(x)=\sum\limits_{Q_{j,k}\subset Q}g^{\epsilon}_{j,k}\Phi^{\epsilon}_{j,k}(x)$$
as the localization of $g(x)$ on $Q$.
Then we restrict the range of frequency by limiting the number of $j$.
In fact, we consider the function
\begin{equation}\label{ml} g_{t,Q}(x)=:\sum_{Q_{j,k}\subset Q:
-\log_{2}|Q|\leq nj\leq
nt-\log_{2}|Q|}g^{\epsilon}_{j,k}\Phi^{\epsilon}_{j,k}(x).
\end{equation}
We obtain three micro-local quantities
about $g_{t,Q}(x)$ by using some basic results in analysis.
See Section 3 for details.

In Section 4, applying the above micro-local analysis of functions
in $P^{\alpha}(\mathbb{R}^{n})$, we give a new wavelet characterization
of this space. As the predual of $Q_{\alpha}(\mathbb{R}^{n})$,
$P^{\alpha}(\mathbb{R}^{n})$ has been studied by many authors.
One method is to define the predual space $P^{\alpha}(\mathbb{R}^{n})$
by a family of Borel measures. See Kalita \cite{K},
Wu-Xie \cite{WX} and Yuan-Sickel-Yang \cite{YSY}.
This idea can result in wavelet characterization of the predual space;
but the predual space with the induced norm is a pseudo-Banach space.
Dafni-Xiao \cite{DX} used a method of Hausdorff capacity to study $P^{\alpha}(\mathbb{R}^{n})$.
L. Peng and Q. Yang defined $P^{\alpha}(\mathbb{R}^{n})$ by the atoms ( see \cite{PY} and \cite{Yang}).
By these methods, $P^{\alpha}(\mathbb{R}^{n})$ are Banach spaces;
but these authors did not consider the wavelet characterization of $P^{\alpha}(\mathbb{R}^{n})$.

Compared with the former results of \cite{K}, \cite{PY}, \cite{Yang} and \cite{YSY},
our result owns the following advantage. Let $f$ be a function in $ P^{\alpha}(\mathbb{R}^{n})$.
Our wavelet characterization indicates clearly that different frequencies
exert different influences  to the $P^{\alpha}$-norm of $f$.  See Theorem
\ref{th4} and Theorem \ref{cha}.
To obtain the Fefferman-Stein type decomposition of
$Q_{\alpha}(\mathbb{R}^{n})$, we need such
a wavelet characterization of $P^{\alpha}(\mathbb{R}^{n})$.

In Section 5, by
the characterization obtained in Section 4 and  the properties of
Meyer wavelets and Daubechies wavelets, we characterize
$P^{\alpha}(\mathbb{R}^{n})$ associated with Riesz transforms. See Theorem \ref{thL1}.
Applying this result and the duality between $P^{\alpha}(\mathbb{R}^{n})$
and $Q_{\alpha}(\mathbb{R}^{n})$, we obtain a Fefferman-Stein type decomposition of
$Q_{\alpha}(\mathbb{R}^{n})$.

We need to point out that our definition of
$Q_{\alpha}(\mathbb{R}^{n})$ is different from the one introduced in
\cite{EJPX}. For non-trivial spaces, the scope of $\alpha$ in \cite{EJPX}
is restricted to $(0,\min\{1,\frac{n}{2}\})$, while the scope in our definition
can be relaxed to $(0,\frac{n}{2})$. More importantly,
when $\alpha\in(0,\min\{1,\frac{n}{2}\})$,
our definition is equivalent to the one in \cite{EJPX}. So the
Fefferman-Stein type decomposition obtained in Section 5 gives a
positive answer to the open problem proposed in \cite{EJPX}.

The rest of this paper is organized as follows. In Section 2, we
state some preliminary notations and lemmas which will be used in
the sequel. In Section 3, we study the micro-local quantities for
the functions in $P^{\alpha}(\mathbb{R}^{n})$. In Section 4, we
obtain a wavelet characterization of $P^{\alpha}(\mathbb{R}^{n})$ by
the micro-local result in Section 3. The Fefferman-Stein type
decomposition of $Q_{\alpha}(\mathbb{R}^{n})$ is obtained in Section
5. In Section 6, we prove the Riesz transformation theorem for space
$P^{\alpha}(\mathbb{R}^{n})$.

\section{Preliminaries}
In this section, we present preliminaries on wavelets,
functions and operators which will be used in this paper.

\subsection{Wavelets and classic function spaces}
In this paper, we use real-valued tensor product wavelets; which can
be regular Daubechies wavelets or classical Meyer wavelets. For simplicity, we denote
by $0$ the zero vector in $\mathbb{R}^{n}$. Let
$E_{n}=\{0,1\}^{n}\backslash \{0\}$. Let $\Phi^{0}(x)$ and
$\Phi^{\epsilon}(x), \epsilon\in E_{n},$ be the scale function and
the wavelet functions, respectively.  If $\Phi^{\epsilon}(x)$ is a
Daubechies wavelet, we assume there exist $m>8n$ and $M\in \mathbb{N}$ such that
\begin{itemize}
\item[(1)] $\forall \epsilon\in\{0,1\}^n$,
$\Phi^{\epsilon}(x)\in C^{m}_{0}([-2^{M},2^{M}]^{n})$;
\item[(2)] $\forall \epsilon\in E_{n}$, $\Phi^{\epsilon}(x)$ has the vanishing moments
up to the order $m-1$.
\end{itemize} For further information about wavelets,
we refer the reader to \cite{Me}, \cite{Woj} and \cite{Yang}.

For $j\in\mathbb{Z}$ and $k=(k_{1}, k_{2},\dots, k_{n})\in\mathbb{Z}^{n}$,
we denote by $Q_{j,k}$ the dyadic cube
$\prod\limits^{n}_{s=1}[2^{-j}k_{s}, 2^{-j}(k_{s}+1)]$ and by $$\Omega=\Big\{Q_{j,k},j\in \mathbb{Z},k\in \mathbb{Z}^{n}\Big\},$$ respectively.
Let $$\Lambda_{n}=\Big\{(\epsilon,j,k),\epsilon\in E_{n},j\in \mathbb{Z},k\in
\mathbb{Z}^{n}\Big\}.$$ For $\epsilon\in \{0,1\}^{n}$, $j\in \mathbb{Z}, k\in \mathbb{Z}^{n}$,
we denote $\Phi^{\epsilon}_{j,k}(x) = 2^{{jn}/{2}}
\Phi^{\epsilon}( 2^{j}x-k) $.

The following result is well-known.
\begin{lemma}\label{le3} {\rm (\cite{Me})}
$\Big\{\Phi^{\epsilon}_{j,k}(x),\ (\epsilon,j,k)\in
\Lambda_{n}\Big\}$ is an orthogonal basis in
$L^{2}(\mathbb{R}^{n})$.
\end{lemma}

Let $f^{\epsilon}_{j,k}= \langle f, \Phi^{\epsilon}_{j,k}\rangle$,
$\forall \epsilon\in \{0,1\}^{n}$ and $k\in \mathbb{Z}^{n}$.
By Lemma \ref{le3}, any $L^{2}$ function $f(x)$ has a wavelet decomposition
$$f(x)=\sum_{(\epsilon,j,k)\in\Lambda_{n}} f^{\epsilon}_{j,k}\Phi^{\epsilon}_{j,k}(x).$$

We list some knowledge on Sobolev spaces and Hardy space. For
$1< p<\infty$, we denote by $p'$ the conjugate index of $p$, that is,
$\frac{1}{p}+\frac{1}{p'}=1$.  For a function space $A$, we denote by $(A)'$ the dual space of $A$.
For $1< p<\infty, r\in \mathbb{R}$ and
Sobolev spaces $W^{r,p}(\mathbb{R}^{n})$, we know that $(W^{r,p}(\mathbb{R}^{n}))'=
W^{-r,p'}(\mathbb{R}^{n})$, see \cite{Me}, \cite{Tr} and
\cite{Yang}.

Let $\chi(x)$  be the characteristic
function of the unit cube $[0,1]^n$. For function
$g(x)=\sum\limits_{(\epsilon,j,k)\in \Lambda_{n}} g^{\epsilon}_{j,k}
\Phi^{\epsilon}_{j,k}(x)$, we have the following characterization,
see \cite{Me}, \cite{Yang} and \cite{YCP}:

\begin{proposition}\label{prop1}
(i) Given $1< p<\infty$ and
$|r|<m$.
$$g(x)=\sum\limits_{(\epsilon,j,k)\in \Lambda_{n}} g^{\epsilon}_{j,k}
\Phi^{\epsilon}_{j,k}(x)\in W^{r,p}(\mathbb{R}^{n})$$ if and only if
\begin{eqnarray*}
&& \Big\|\Big(\sum\limits_{(\epsilon,j,k)\in
\Lambda_{n}} 2^{2j(r+\frac{n}{2})}|
g^{\epsilon}_{j,k}|^{2}\chi(2^{j}x-k)\Big)^{\frac{1}{2}}\Big\|_{L^{p}}<\infty.
\end{eqnarray*}
(ii) $g(x)=\sum\limits_{(\epsilon,j,k)\in \Lambda_{n}}
g^{\epsilon}_{j,k} \Phi^{\epsilon}_{j,k}(x)\in H^{1}(\mathbb{R}^{n})
$ if and only if
$$\Big\|\Big(\sum\limits_{(\epsilon,j,k)\in
\Lambda_{n}} 2^{nj}|
g^{\epsilon}_{j,k}|^{2}\chi(2^{j}x-k)\Big)^{\frac{1}{2}}\Big\|_{L^{1}}<\infty.$$
\end{proposition}

\subsection{$Q$-spaces} We know that
$Q_{\alpha}(\mathbb{R}^{n})= BMO(\mathbb{R}^{n})$ for $\alpha\leq 0$. Further,
It is easy to see that the $Q$-spaces  defined in (\ref{eq:111}) are trivial for $\alpha\geq 1$.
In fact, for $\alpha\geq 1$
or $\alpha>\frac{n}{2}$, there are
only constants in $Q_{\alpha}(\mathbb{R}^{n})$ defined in (\ref{eq:111}).

To get rid of the restriction $\alpha\geq 1$,
we introduce a new definition of $Q$-spaces such that the $Q$-spaces are not
trivial for $1\leq \alpha\leq \frac{n}{2}$.

For $\alpha \in \mathbb{R}$,  denote by
$$f_{\alpha,Q}=|Q|^{-1} \int_{Q}
(-\Delta)^{\frac{\alpha}{2}}f(x) dx$$
 the mean value of the function
$(-\Delta)^{\frac{\alpha}{2}}f$ on the cube $Q$. For $\alpha\in \mathbb{R}$, let
$$B_{\alpha,Q}f=|Q|^{\frac{\alpha}{n}} \bigg( |Q|^{-1} \int_{Q}
|(-\Delta)^{\frac{\alpha}{2}}f(x)- f_{\alpha,Q}|^{2} dx\bigg) ^{\frac{1}{2}}.$$
$Q$-spaces $Q_{\alpha}(\mathbb{R}^{n})$ and $Q^{0}_{\alpha}(\mathbb{R}^{n})$ are defined as follows.

\begin{definition} \label{def:1} Given $\alpha\in[0,\frac{n}{2}]$.
\begin{itemize}
\item [(i)] $Q_{\alpha}(\mathbb{R}^{n})$  is a space defined as the set of all measurable functions $f$ with
$$\sup\limits_{Q}B_{\alpha,Q}f<\infty,$$ where the supremum is taken over all cubes $Q$.
\item [(ii)] $Q^{0}_{\alpha}(\mathbb{R}^{n})$ is a space defined as the set of all measurable functions $f\in Q_{\alpha}(\mathbb{R}^{n})$ and satisfying

$$\left\{\begin{array}{cl}
&\lim\limits_{|Q|\rightarrow 0 } B_{\alpha,Q}f=0,\\
&\lim\limits_{|Q|\rightarrow  \infty} B_{\alpha,Q}f=0,
\end{array}\right.$$
where the supremum and the limit are taken over all cubes $Q$.
\end{itemize}

\end{definition}

\begin{remark}
If $\alpha=\frac{n}{2}$, then $Q_{\frac{n}{2}}(\mathbb{R}^{n})=\dot{B}^{\frac{n}{2},2}_{2}(\mathbb{R}^{n})$.
For $1\leq \alpha\leq \frac{n}{2}$, the $Q$-spaces in Definition \ref{def:1}
are not trivial. Further, for other indices $\alpha$,
the corresponding $Q_{\alpha}(\mathbb{R}^{n})$ coincide with ones defined in \cite{EJPX}.
So $Q_{\alpha}(\mathbb{R}^{n})$, defined in Definition \ref{def:1}, is a generalization of
$Q$-spaces defined in (\ref{eq:111}).
\end{remark}

For $|\alpha|<m$, $Q\in\Omega$ and function
$f(x) = \sum\limits_{(\epsilon,j,k)\in \Lambda_{n}} f^{\epsilon}_{j,k}
\Phi^{\epsilon}_{j,k}(x)$, let
$$C_{\alpha,Q}f=|Q|^{\frac{\alpha}{n}-\frac{1}{2}}\Big(\sum\limits_{Q_{j,k}\subset
Q} 2^{2j\alpha}|f^{\epsilon}_{j,k}|^{2}\Big)^{\frac{1}{2}}.$$  By wavelet characterization
of Sobolev space $W^{s,2}(\mathbb{R}^{n})$, we get the following wavelet
characterization of $Q$-spaces, cf  \cite{YSY}:

\begin{proposition}\label{pro:1}
Given $0\leq \alpha \leq \frac{n}{2}$.
\begin{itemize}
\item[(i)] $f(x)=\sum\limits_{(\epsilon,j,k)\in \Lambda_{n}}
f^{\epsilon}_{j,k} \Phi^{\epsilon}_{j,k}(x) \in Q_{\alpha}(\mathbb{R}^{n})$ if and only if
$$\sup\limits_{Q\in\Omega}C_{\alpha,Q}f<\infty. $$

\item[(ii)] $f(x)=\sum\limits_{(\epsilon,j,k)\in \Lambda_{n}}
f^{\epsilon}_{j,k} \Phi^{\epsilon}_{j,k}(x) \in Q^{0}_{\alpha}(\mathbb{R}^{n})$ if and only if
\begin{equation}\label{eq:r}
\left\{\begin{array}{cl}
&\sup\limits_{Q\in\Omega} C_{\alpha,Q}f<\infty,\\
&\lim\limits_{Q\in \Omega,\, |Q|\rightarrow 0} C_{\alpha,Q}f=0,\\
&\lim\limits_{Q\in \Omega,\, |Q|\rightarrow  \infty} C_{\alpha,Q}f=0.
\end{array}\right.
\end{equation}
\end{itemize}
\end{proposition}

By Propositions \ref{prop1} and \ref{pro:1},  we may identify a function
$$g(x)=\sum\limits_{(\epsilon,j,k)\in \Lambda_{n}}
g^{\epsilon}_{j,k} \Phi^{\epsilon}_{j,k}(x)$$ with
the sequence $\{g^{\epsilon}_{j,k}\}_{(\epsilon,j,k)\in \Lambda_{n}}.$
%For example, the notation $\{f^{\epsilon}_{j,k} \}_{(\epsilon,j,k)\in \Lambda_{n}}
%\in \tilde{Q}_{\alpha}(\mathbb{R}^{n})$ is used to indicate $f(x)=\sum\limits_{(\epsilon,j,k)\in \Lambda_{n}}
%f^{\epsilon}_{j,k} \Phi^{\epsilon}_{j,k}(x) \in Q_{\alpha}(\mathbb{R}^{n})$ and
%the notation $ \{f^{\epsilon}_{j,k} \}_{(\epsilon,j,k)\in \Lambda_{n}}
%\in \tilde{Q}^{0}_{\alpha}(\mathbb{R}^{n})$ is similarly used to indicate
%$f(x)=\sum\limits_{(\epsilon,j,k)\in \Lambda_{n}}
%f^{\epsilon}_{j,k} \Phi^{\epsilon}_{j,k}(x) \in Q^{0}_{\alpha}(\mathbb{R}^{n})$.

\subsection{Calder\'on-Zygmund operators}
In this subsection, we introduce some preliminaries about Calder\'on-Zygmund
operators, see \cite{Me} and \cite{Stein}. For $x\neq y$, let
$K(x,y)$ be a smooth function such that
\begin{equation}\label{eq2}
|\partial ^{\alpha}_{x}\partial ^{\beta}_{y} K(x,y)| \leq
\frac{C}{|x-y|^{n+|\alpha|+|\beta|}}, \forall |\alpha|+ |\beta|\leq
N_{0},
\end{equation}
where $N_{0}$ is a large enough constant and $N_{0}\leq m$.

A linear operator $T$ is said to be a Calder\'on-Zygmund operator in $CZO(N_{0})$
if
\begin{itemize}
\item[(1)]
$T$ is continuous from $C^{1}(\mathbb{R}^{n})$ to $(C^{1}(\mathbb{R}^{n}))'$;
\item[(2)] There exists a kernel $K(x,y)$ satisfying (\ref{eq2}) and
for $x\notin {\rm supp} f(x)$,
$$Tf(x)=\int K(x,y) f(y) dy;$$
\item[(3)] $Tx^{\alpha}=T^{*}x^{\alpha}=0, \forall \alpha \in \mathbb{N}^{n}$
and $|\alpha|\leq N_{0}$.\end{itemize} %We denote by $CZO(N_{0})$ the class of operators $T$ satisfying the conditions (1), (2), (3).

\begin{remark}
The values of $K(x,y)$ in (\ref{eq2}) have not been defined for $x=y$.
According to Schwartz kernel theorem, the kernel $K(x,y)$
of a linear continuous operator $T$ is only a distribution in $S'(\mathbb{R}^{2n})$.
\end{remark}
Let $\{\Phi^{\epsilon}_{j,k}(x)\}_{(\epsilon,j,k)\in\Lambda_{n}}$ be a
sufficient regular wavelet basis.
For $ (\epsilon,j,k),$ $(\epsilon',j',k')\in \Lambda_{n}$, we denote
$$a^{\epsilon,\epsilon'}_{j,k,j',k'}= \Big\langle K(\cdot,\cdot),
\Phi^{\epsilon}_{j,k}(\cdot) \Phi^{\epsilon'}_{j',k'}(\cdot)\Big\rangle.$$

\begin{lemma}({\rm \cite{Me}}) (i) If~$T\in CZO(N_{0})$, for all $
(\epsilon,j,k)$ and $(\epsilon',j',k')\in \Lambda_{n},$ the coefficients
$a^{\epsilon,\epsilon'}_{j,k,j',k'}$ satisfy that
\begin{equation}\label{eq1}
|a^{\epsilon,\epsilon'}_{j,k,j',k'}| \leq C
2^{-|j-j'|(\frac{n}{2}+N_{0})}
\Big(\frac{2^{-j}+2^{-j'}}{2^{-j}+2^{-j'}
+|k2^{-j}-k'2^{-j'}|}\Big)^{n+N_{0}}.
\end{equation}

(ii) If $\{a^{\epsilon,\epsilon'}_{j,k,j',k'}\}_{
(\epsilon,j,k), (\epsilon',j',k')\in \Lambda_{n}}$ satisfies (\ref{eq1}), then
$$K(x,y)=\sum\limits_{
(\epsilon,j,k),(\epsilon',j',k')\in \Lambda_{n}}
a^{\epsilon,\epsilon'}_{j,k,j',k'}\Phi^{\epsilon}_{j,k}(x)
\Phi^{\epsilon'}_{j',k'}(y)$$ in the sense of distributions. Further, for any
$0<\delta<N_0$, we have $T\in CZO(N_{0}-\delta)$.

\end{lemma}

At the end of this subsection, we list a variant result about
the continuity of Calder\'on-Zygmund operators on Sobolev spaces ( see
also [12]).

For all $(\epsilon,j,k)\in \Lambda_{n}$,
denote
$$\tilde{g}^{\epsilon}_{j,k} =\sum \limits_{
(\epsilon',j',k')\in \Lambda_{n}} a^{\epsilon,\epsilon'}_{j,k,j',k'}
g^{\epsilon'}_{j',k'}.$$ We have

\begin{lemma} Given $|r|<s\leq m$ and $1< p<\infty$. If\ $
\forall\, (\epsilon,j,k),\, (\epsilon',j',k')\in \Lambda_{n},$
$$|a^{\epsilon,\epsilon'}_{j,k,j',k'}|
\leq C 2^{-|j-j'|(\frac{n}{2}+s)}
\Big(\frac{2^{-j}+2^{-j'}}{2^{-j}+2^{-j'}
+|k2^{-j}-k'2^{-j'}|}\Big)^{n+s},$$  then
\begin{eqnarray*}
&&\int \Big(\sum\limits_{(\epsilon,j,k)\in \Lambda_{n}}
2^{j(n+2r)} |\tilde{g}^{\epsilon}_{j,k}|^{2} \chi (2^{j}x-k)
\Big)^{\frac{p}{2}} dx\\
&\leq& C \int \Big(\sum \limits_{(\epsilon,j,k)\in \Lambda_{n}}
2^{j(n+2r)} |g^{\epsilon}_{j,k}|^{2} \chi (2^{j}x-k) \Big)^{\frac{p}{2}} dx.\end{eqnarray*}
\end{lemma}

\subsection{Generalized Hardy spaces}
In \cite{PY} and \cite{Yang}, L. Peng and Q. Yang used atoms to define the predual
spaces of the $Q$-spaces. Below we introduce the
standard atoms, the wavelet atoms and the related generalized Hardy
spaces:
\begin{definition}
Given $0\leq\alpha<\frac{n}{2}$.
\begin{itemize}
\item[(i)] A distribution $g(x)$ is an $(\alpha, 2)-$atom on a cube $Q$
if
\begin{itemize}
\item[(1)] $\|(-\Delta)^{-\frac{\alpha}{2}}g\|_{L^{2}} \leq
|Q|^{-\frac{1}{2}+\frac{\alpha}{n}} $,
\item[(2)] ${\rm supp} g(x) \subset Q,$
\item[(3)] $\int x^{\beta} g(x)dx=0,\forall |\beta|\leq |\alpha|$.
\end{itemize}

\item[(ii)] A distribution $f(x)$ belongs to a Hardy space
$P^{\alpha}(\mathbb{R}^{n})$ if
$$f(x)= \sum\limits_{u\in\mathbb{Z}}
\lambda_{u} g_{u}(x),$$
 where $\{\lambda_{u}\}_{u\in\mathbb{Z}}\in l^{1}$ and $g_{u}(x)$
are $( \alpha, 2)-$atoms.
\end{itemize}
\end {definition}

\begin{definition}Given $0\leq\alpha<\frac{n}{2}$.
\begin{itemize}
\item[(i)] A distribution $g(x)=\sum\limits_{\epsilon\in E_{n},Q_{j,k}\subset Q}
g^{\epsilon}_{j,k} \Phi^{\epsilon}_{j,k}(x)$ is a $(\alpha,2)-$wavelet atom
on a dyadic cube $Q$ if
$$\Big(\sum\limits_{(\epsilon,j,k)\in \Lambda_{n}} 2^{-2j\alpha} |g^{\epsilon}_{j,k}|^{2}\Big)
^{\frac{1}{2}} \leq |Q|^{\frac{\alpha}{n}-\frac{1}{2}} .$$

\item[(ii)] A distribution $f(x)$ belongs to a Hardy space
$P^{\alpha}_{w}(\mathbb{R}^{n})$ if
$$f(x)= \sum\limits_{u\in\mathbb{Z}}
\lambda_{u} g_{u}(x),$$ where $\{\lambda_{u}\}_{u\in\mathbb{Z}}\in l^{1}$ and
$g_{u}(x)$ are $(\alpha, 2)-$wavelet atoms.\end{itemize}
\end{definition}

In fact, $P^{\alpha}(\mathbb{R}^{n})$ and $P^{\alpha}_{w}(\mathbb{R}^{n})$ are the same spaces
and Calder\'on-Zygmund operators are bounded on $P^{\alpha}(\mathbb{R}^{n})$. The proof of the
following lemma can be found in \cite{PY} and \cite{Yang}.
\begin{proposition}\label{pro:CZ}
Given $0\leq \alpha< \frac{n}{2}$.
\begin{itemize}
\item[(i)] $P^{\alpha}(\mathbb{R}^{n})=P^{\alpha}_{w}(\mathbb{R}^{n})$.

\item[(ii)] Any operator $T\in CZO(N_{0})$ is bounded
on $P^{\alpha}(\mathbb{R}^{n})$.
\end{itemize}
\end{proposition}

For $\alpha=\frac{n}{2}$, define
$$P^{\frac{n}{2}}(\mathbb{R}^{n})=:\dot{B}^{-\frac{n}{2},2}_{2}(\mathbb{R}^{n}).$$
Applying the same ideas in \cite{PY}, \cite{S}, \cite{Stegenga} and
\cite{YSY}, we have the following duality relation.

\begin{proposition} \label{pro:dual}
Given $0\leq \alpha\leq \frac{n}{2}$.
\begin{itemize}
\item[(i)] $(P^{\alpha}(\mathbb{R}^{n}))'=Q_{\alpha}(\mathbb{R}^{n});$

\item[(ii)] $(Q^{0}_{\alpha}(\mathbb{R}^{n}))'=P^{\alpha}(\mathbb{R}^{n}).$
\end{itemize}
\end{proposition}

\section{Micro-local quantities for $P^{\alpha}(\mathbb{R}^{n})$}

If $\alpha=0$, $P^{0}(\mathbb{R}^{n})=H^{1}(\mathbb{R}^{n})$. If
$\alpha=\frac{n}{2}$,  $P^{\frac{n}{2}}(\mathbb{R}^{n})=\dot{B}^{-\frac{n}{2},2}_{2}(\mathbb{R}^{n}).$
It is well known that the norms of
$P^{0}(\mathbb{R}^{n})$ and $P^{\frac{n}{2}}(\mathbb{R}^{n})$ depend only on the
$L^{p}(l^{2})$-norms of function series $\{f_{j}=Q_{j}f\}_{j\in \mathbb{Z}}$ for $p=1$ and $p=2$, respectively.
For the case $0<\alpha<\frac{n}{2}$, the situation is complicated.

In this section, we use wavelets to analyze the micro-local structure of $P^{\alpha}(\mathbb{R}^{n})$.
First, we present a theorem on conditional maximum value in Subsection 3.1.
Then we consider the micro-local quantities in Subsection 3.2.

\subsection{Conditional maximum value for non-negative sequence}
For $u\in \mathbb{N}$, we denote
$$\left\{\begin{array}{cl}
\Lambda_{u,n}=& \Big\{0,1,\cdots,
2^{u}-1\Big\}^{n};\\
G_{u,n}=&\Big\{(\epsilon,s,v),\epsilon\in E_{n}, 0\leq
s\leq u, v\in \Lambda_{s,n}\Big\}.
\end{array}\right.$$

$\forall j\in \mathbb{Z}$, $k\in \mathbb{Z}^n$, $t\in \mathbb{N}$
and  sequence $\tilde{g}^t_{j,k}=\{g^{\epsilon}_{j+s, 2^s k+u}\}_{
(\epsilon,s,u)\in G_{t,n}}$, we define
\begin{definition}
We call $\tilde{g}^{t}_{j,k}=\{g^{\epsilon}_{j+s, 2^s k+u}\}_{(\epsilon,s,u)
\in G_{t,n}}$ a non-negative sequence if
$\tilde{g}^{t}_{j,k}$ satisfies \begin{equation}\label{eq:non}\forall \,\,
(\epsilon,s,u)\in G_{t,n},\,\, g^{\epsilon}_{j+s, 2^s k+u}\geq 0.
\end{equation}
\end{definition}

For  a non-negative sequence $\tilde{g}^{t}_{j,k}$,
we find the maximum value of the following quantities:
\begin{equation}\label{eq:maxm}\tau_{f^{t}_{j,k}, \tilde{g}^t_{j,k}}=
\sum\limits_ {(\epsilon,s,u)\in G_{t,n}}    f^t_{j,k} \tilde{g}^t_{j,k},
\end{equation}
where non-negative sequence $f^t_{j,k}=\{f^{\epsilon}_{j+s, 2^s k+u}\}_{(\epsilon,s,u)\in G_{t,n}}$
satisfies the following $\sum\limits_{0\leq s\leq t} 2^{ns}$ restricted conditions
\begin{equation}\label{eq:max}
\left\{\begin{array}{cll}
2^{n(j+t)}\sum\limits_{\epsilon\in E_n} (f^{\epsilon}_{j+t, 2^t k+u})^2&\leq 1,&
\forall u\in \Lambda_{t,n};\\
2^{n(j+t-1)}\sum\limits_{(\epsilon,s,v)\in G_{1,n}} 2^{2s\alpha}(f^{\epsilon}_{j+t-1+s, 2^{s}(2^{t-1} k+u)+v})^2&\leq 1,&
\forall u\in \Lambda_{t-1,n};\\
2^{n(j+t-2)}\sum\limits_{(\epsilon,s,v)\in G_{2,n}} 2^{2s\alpha}(f^{\epsilon}_{j+t-2+s, 2^{s}(2^{t-2} k+u)+v})^2&\leq 1,&
\forall u\in \Lambda_{t-2,n};\\
\cdots&\leq 1,&\cdots;\\
2^{nj}\sum\limits_{(\epsilon,s,v)\in G_{t,n}} 2^{2s\alpha}(f^{\epsilon}_{j+s, 2^{s} k+v})^2&\leq 1.&
\end{array}
\right.
\end{equation}

There exist $(2^{n}-1)\sum\limits_{0\leq s\leq t} 2^{ns}$ elements in $G_{t,n}$. We can see that
$f^t_{j,k}$ is a sequence, where the number of nonnegative terms
is at most $(2^{n}-1)\sum\limits_{0\leq s\leq t} 2^{ns}$.
\begin{definition}
$\forall j\in \mathbb{Z}$, $k\in \mathbb{Z}^n$, $t\in \mathbb{N}$,
we call $f^t_{j,k}=\{f^{\epsilon}_{j+s, 2^s k+u}\}_{(\epsilon,s,u)\in G_{t,n}}\in F^t_{j,k}$, if
$f^t_{j,k}$ is a non-negative sequence  satisfying the conditions in (\ref{eq:max}).
\end{definition}

According to the basic results in analysis, we have:
\begin{theorem}\label{th:max}
Given $0\leq \alpha<\frac{n}{2}$ and $t\geq 0$. For any non-negative sequence
$\tilde{g}^t_{j,k}=\{g^{\epsilon}_{j+s, 2^s k+u}\}_{(\epsilon,s,u)\in G_{t,n}}$,
there exists at least a sequence $\bar{ f}^{t}_{j,k}
=\{\tilde{f}^{\epsilon}_{j+s, 2^s k+u}\}_{(\epsilon,s,u)\in G_{t,n}}\in F^t_{j,k}$
such that
$$ \tau_{\bar{f}^{t}_{j,k},\tilde{g}^{t}_{j,k} }=
\max\limits_{f^{t}_{j,k}\in F^t_{j,k}} \tau_{f^{t}_{j,k},\tilde{g}^{t}_{j,k} }.$$
\end{theorem}
\begin{proof}
The $(2^{n}-1)\sum\limits_{0\leq s\leq t} 2^{ns}$ variables
$\{f^{\epsilon}_{j+s, 2^s k+u}\}_{(\epsilon,s,u)\in G_{t,n}}$
of the sequence $f^t_{j,k}$ are restricted in a closed domain,
so the conclusion is obvious.
\end{proof}

\subsection{Micro-local quantities in $P^{\alpha}(\mathbb{R}^{n})$ }

From Proposition \ref{pro:dual}, we know that $(Q^{0}_{\alpha}(\mathbb{R}^{n}))'=P^{\alpha}(\mathbb{R}^{n})$.
To prove  a function $g\in P^{\alpha}(\mathbb{R}^{n})$, we only need to consider
$\sup\limits_{\|f\|_{Q^{0}_{\alpha}}\leq1}\langle f,\ g\rangle,$
where the supremum is taken over all $f\in Q^{0}_{\alpha}$ with $\|f\|_{Q^{0}_{\alpha}}\leq1$. However, by this method,
we can not know the micro-local structure  of $g(x)$ in details.

%To obtain the micro-local information of functions in $P^{\alpha}(\mathbb{R}^{n})$, we apply a new method.

To delete this shortage, we introduce a new method.
Let
$$g(x) = \sum\limits_{(\epsilon,j,k)\in \Lambda_n}
g^{\epsilon}_{j,k} \Phi^{\epsilon}_{j,k}(x).$$
 We localize $g(x)$ by restricting its wavelet coefficients
$g^{\epsilon}_{j,k}$ such that $Q_{j,k}\subset Q$. Then we limit the range of frequencies and analyze
its micro-local information.

For this purpose, we analyze
the function  $g_{t,Q}(x)$ defined in (\ref{ml}).
For such a $g_{t,Q}$,  the number of $(\epsilon,j,k)$ such that
$g^{\epsilon}_{j,k}\neq 0$ is at most $(2^{n}-1)\sum\limits_{0\leq s\leq t} 2^{ns}$.
We study micro-local functions $g_{t,Q}$ in
$P^{\alpha}(\mathbb{R}^{n})$ and obtain  three kinds of
micro-local quantities.

For all $t,j\in \mathbb{Z}, k\in
\mathbb{Z}^{n}$ and $t\geq 0$, we consider the series
$$g^{t}_{j,k}=\Big\{g^{\epsilon}_{j+s,2^{s}k+v}, \epsilon\in E_{n}, 0\leq
s\leq t, v\in \Lambda_{s,n}\Big\}.$$
 Denote
\begin{equation}\label{eq3} g^{t}_{j,k}(x)=\sum\limits_{(\epsilon,s,u)\in
G_{t,n}} g^{\epsilon}_{j+s, 2^{s}k+u}\Phi^{\epsilon}_{j+s,
2^{s}k+u}(x).\end{equation}
Because there is a one-to-one relation between the
sequences $g^{t}_{j,k}$ and the function $g^{t}_{j,k}(x)$,
sometimes, we do not distinguish them.

For simplicity, we suppose that our functions are
real-valued. Let
$$\left\{\begin{array}{cl}
f(x) =& \sum\limits_{(\epsilon,j,k)\in\Lambda_n}
f^{\epsilon}_{j,k} \Phi^{\epsilon}_{j,k}(x);\\
g(x) =&
\sum\limits_{(\epsilon,j,k)\in\Lambda_n} g^{\epsilon}_{j,k}
\Phi^{\epsilon}_{j,k}(x).
\end{array}\right.$$

If $\langle f,g\rangle$ and
$\sum\limits_{(\epsilon,j,k)\in\Lambda_n} f^{\epsilon}_{j,k}
g^{\epsilon}_{j,k}$ are well defined, then we have
\begin{equation}\label{3.1.1}
\tau_{f,g}=: \langle f,\ g\rangle=
\sum\limits_{(\epsilon,j,k)\in\Lambda_n} f^{\epsilon}_{j,k}
g^{\epsilon}_{j,k}.
\end{equation}
To compute $\max\limits_{\|f\|_{Q^0_{\alpha}}\leq 1}\tau_{f,g^t_{j,k}}$,
according to (\ref{3.1.1}),
we can restrict $f$ to the function
$$f^{t}_{j,k}(x)=
\sum\limits_{(\epsilon,s,u)\in G_{t,n}} f^{\epsilon}_{j+s,2^{s}k+u}
\Phi^{\epsilon}_{j+s,2^{s}k+u}(x)$$
 with $\|f^{t}_{j,k}\|_{Q^0_{\alpha}}\leq 1$.
The number of $(\epsilon,j,k)$ such that
$f^{\epsilon}_{j,k}\neq 0$ is at most $(2^{n}-1)\sum\limits_{0\leq s\leq t} 2^{ns}$. That is to say, applying (\ref{3.1.1}),
we transfer the problem to finding out the supremum on infinite restricted
conditions to a maximum value problem on $ \sum\limits_{s=0}^t 2^{ns} $  restricted conditions on the series of quantities $\{f^{\epsilon}_{j+s,2^sk+u}\}_{
(\epsilon,s,u)\in G_{t,n}}.$

Based on Theorem \ref{th:max}, we begin to consider the
micro-local quantities of $g^{t}_{j,k}$ in $P^{\alpha}(\mathbb{R}^{n})$.
\begin{theorem}\label{th4}
Given $0<\alpha<\frac{n}{2}$ and $t\geq 0$. Let %$g(x) = \sum\limits_{(\epsilon,j,k)\in \Lambda_n}
%g^{\epsilon}_{j,k} \Phi^{\epsilon}_{j,k}(x)\in P^{\alpha}(\mathbb{R}^{n})$ and
$g^{t}_{j,k}$ be the function defined by (\ref{eq3})
and $\|g^{t}_{j,k}\|_{P^{\alpha}}>0$.
\begin{itemize}
\item[(i)] There exists a function
$$S f^{t}_{j,k}(x)=
\sum\limits_{(\epsilon,s,u)\in G_{t,n}}S^{t}_{j} f^{\epsilon}_{j+s,2^{s}k+u}
\Phi^{\epsilon}_{j+s,2^{s}k+u}(x)$$ with
$\|S^{t}_{j} f^{t}_{j,k}\|_{Q^0_{\alpha}}\leq 1$
such that
$$\max\limits_{\|f\|_{Q^0_{\alpha}}\leq 1} \tau_{f,g^{t}_{j,k} }=\sum\limits_{(\epsilon,s,u)\in G_{t,n}}
S^{t}_{j} f^{\epsilon}_{j+s,2^{s}k+u}\cdot g^{t}_{j+s, 2^{s}k+u}.$$

\item[(ii)] There exists a positive number $P^{t}_{j}g^{t}_{j,k}$ which is defined
by the absolute values of wavelet coefficient of $g^{t}_{j,k}$ such that
$$P^{t}_{j}g^{t}_{j,k}=\|g^{t}_{j,k}\|_{P^{\alpha}}=
\max\limits_{\|f\|_{Q^0_{\alpha}}\leq 1} \tau_{f,g^{t}_{j,k} }=
\tau_{S f^{t}_{j,k},g^{t}_{j,k} }.$$

\item[(iii)] There exists a sequence $\{Q^{t}_{j}g^{\epsilon}_{j,k}\}_{\epsilon\in E_{n}}$
such that $\sum\limits_{\epsilon\in E_{n}} Q^{t}_{j}g^{\epsilon}_{j,k}\Phi^{\epsilon}_{j,k}(x)$
has the same norm in $P^{\alpha}(\mathbb{R}^{n})$ as $g^{t}_{j,k}$ does.
\end{itemize}
\end{theorem}

\begin{proof}
For $g^t_{j,k}=\{g^{\epsilon}_{j+s, 2^s k+u}\}_{(\epsilon,s,u)\in G_{t,n}}$,
denote $\tilde{g}^t_{j,k}=\{|g^{\epsilon}_{j+s, 2^s k+u}|\}_{(\epsilon,s,u)\in G_{t,n}}$.
Denote
$$G^{t,j,k}_{g}=\Big\{(\epsilon,s,u)\in G_{t,n}, g^{\epsilon}_{j+s, 2^{s}k+u}\neq 0\Big\}.$$
For $f^{t}_{j,k}(x)=
\sum\limits_{(\epsilon,s,u)\in G_{t,n}} f^{\epsilon}_{j+s,2^{s}k+u}
\Phi^{\epsilon}_{j+s,2^{s}k+u}(x)$, define
$$f^{\epsilon,g}_{j+s, 2^{s}k+u}=\left\{\begin{array}{cl}
|f^{\epsilon}_{j+s, 2^{s}k+u}|\cdot
|g^{\epsilon}_{j+s, 2^{s}k+u}|^{-1}  \overline{g^{\epsilon}_{j+s, 2^{s}k+u}},
&(\epsilon,s,u)\in G_{t,n};\\
0, & (\epsilon,s,u)\notin G_{t,n}.
\end{array}\right.$$
We denote by $F^{t,j,k}_{g}$ the set
$$\Big\{f^{t}_{j,k}:\ f^{t}_{j,k}(x)=
\sum\limits_{(\epsilon,s,u)\in G_{t,n}} f^{\epsilon,g}_{j+s,2^{s}k+u}
\Phi^{\epsilon}_{j+s,2^{s}k+u}(x) \text{ and } \|f^{t}_{j,k}\|_{Q^0_{\alpha}}\leq 1\Big\}.$$

By (ii) of Proposition \ref{pro:1}, we have
%$$\|f^{t}_{j,k}\|_{Q^0_{\alpha}}\leq 1 \text{ if and only if} \tilde{f}^t_{j,k}\in F^t_{j,k}.$$ Hence
\begin{equation}\label{eq:eq111}
\max\limits_{\|f^{t}_{j,k}\|_{Q^0_{\alpha}\leq 1}} \tau_{\tilde{f}^{t}_{j,k},g^{t}_{j,k} }
=  \max\limits_{f^{t}_{j,k}\in F^{t,j,k}_g} \tau_{f^{t}_{j,k},g^{t}_{j,k} }
=  \max\limits_{\tilde{f}^{t}_{j,k}\in F^{t,j,k}_g } \tau_{\tilde{f}^{t}_{j,k},\tilde{g}^{t}_{j,k} }.
\end{equation}

By the wavelet characterization of $Q_{\alpha}(\mathbb{R}^{n})$,
the condition $\|\tilde{f}^{t}_{j,k}\|_{Q^{0}_{\alpha}\leq 1}$ is equivalent to the conditions
(\ref{eq:max}). Further, for fixed $\tilde{g}^{t}_{j,k}$,
because of (\ref{3.1.1}), if $(\epsilon,s,u)\in G_{t,n}$ and
$(\epsilon,s,u)\notin G^{t,j,k}_{g}$, the coefficients
$f^{\epsilon}_{j+s,2^{s}k+u}$ make no contribution to
$\tau_{\tilde{f}^{t}_{j,k},\tilde{g}^{t}_{j,k}}$.  We get
$$\max\limits_{\tilde{f}^{t}_{j,k}\in F^{t,j,k}_g } \tau_{\tilde{f}^{t}_{j,k},
\tilde{g}^{t}_{j,k} }= \max\limits_{\tilde{f}^{t}_{j,k}\in F^t_{j,k}}
\tau_{\tilde{f}^{t}_{j,k},\tilde{g}^{t}_{j,k} }.$$

According to Theorem \ref{th:max},
there exists at least a sequence
$$\bar{ f}^{t}_{j,k}
=\{\tilde{f}^{\epsilon}_{j+s, 2^s k+u}\}_{(\epsilon,s,u)\in G_{t,n}}\in F^t_{j,k}$$
such that
\begin{equation}\label{eq:m}\tau_{\bar{f}^{t}_{j,k},\tilde{g}^{t}_{j,k} }=
\max\limits_{f^{t}_{j,k}\in F^t_{j,k}} \tau_{f^{t}_{j,k},\tilde{g}^{t}_{j,k} }.
\end{equation}

Let $S f^{t}_{j,k}(x)=
\sum\limits_{(\epsilon,s,u)\in G_{t,n}}S^{t}_{j} f^{\epsilon}_{j+s,2^{s}k+u}
\Phi^{\epsilon}_{j+s,2^{s}k+u}(x)$, where
$$S^{t}_{j} f^{\epsilon}_{j+s,2^{s}k+u}=\left\{\begin{array}{cl}
\tilde{f}^{\epsilon}_{j+s, 2^{s}k+u}
|g^{\epsilon}_{j+s, 2^{s}k+u}|^{-1}  \overline{g^{\epsilon}_{j+s, 2^{s}k+u}},
&\forall (\epsilon,s,u)\in G_{t,n};\\
0, &\forall (\epsilon,s,u)\notin G_{t,n} .
\end{array}\right.$$
According to (\ref{3.1.1}) and (\ref{eq:eq111}), $S f^{t}_{j,k}(x)$ satisfies (i).

Let $P^{t}_{j}g^{t}_{j,k}=\tau_{\bar{f}^{t}_{j,k},\tilde{g}^{t}_{j,k} }$.
According to the last equality in (\ref{eq:eq111}),
$P^{t}_{j}g^{t}_{j,k}$ is defined by the absolute values of wavelet
coefficients of function $g^{t}_{j,k}$.
According to (\ref{3.1.1}), (\ref{eq:eq111}) and (\ref{eq:m}),
$P^{t}_{j}g^{t}_{j,k}$ satisfies (ii).

Denote
$$Q^{t}_{j}g^{\epsilon}_{j,k}=\left\{\begin{array}{cl}
2^{\frac{n}{2}(j-1)} P^{t}_{j}g^{t}_{j,k}, &\mbox{ if } \sum\limits_{\epsilon\in E_n} |g^{\epsilon}_{j,k}|=0;\\
2^{\frac{n}{2}j}P^{t}_{j}g^{t}_{j,k} \Big(\sum\limits_{\epsilon\in E_n} |g^{\epsilon}_{j,k}|^2\Big)^{-\frac{1}{2}}g^{\epsilon}_{j,k} ,
&\mbox{ if } \sum\limits_{\epsilon\in E_n} |g^{\epsilon}_{j,k}|\neq 0.
\end{array}\right.$$
Applying (ii) of Proposition \ref{pro:1} again, we know that
$\{Q^{t}_{j}g^{\epsilon}_{j,k}\}_{\epsilon\in E_n}$ satisfies the condition (iii).
\end{proof}

\begin{remark}
For $\alpha=0$ and $\alpha=\frac{n}{2}$, if we deal with $P^{t}_{j}g^{t}_{j,k}$ in a similar way, then:
\begin{itemize}
\item [(i)] For $\alpha=0$, according to the wavelet characterization of Hardy space $H^1$ in  \cite{Me},
$P^{t}_{j}g^{t}_{j,k}$ can be equivalent to  $$\Big\|\Big(\sum\limits_{(\epsilon,s,u)\in G_{t,n}} 2^{n(j+s)}
|g^{\epsilon}_{j+s,2^sk+u}|^{2} \chi(2^{j+s}\cdot-2^s k-u)\Big)^{\frac{1}{2}}\Big\|_{L^1}.$$
\item [(ii)] For $\alpha=\frac{n}{2}$, $P^{t}_{j}g^{t}_{j,k}$ can be written as
$\Big(\sum\limits_{(\epsilon,s,u)\in G_{t,n}} 2^{-nj} |g^{\epsilon}_{j,k}|^{2}\Big)^{\frac{1}{2}}$.
\end{itemize}
But for $0<\alpha<\frac{n}{2}$, $P^{t}_{j}g^{t}_{j,k}$ can not be expressed in an explicit way.
Luckily, the three parts
$$\left\{\begin{array}{cl}
&\{Q^{t}_{j}g^{\epsilon}_{j,k}\}_{\epsilon\in E_{n}},\\
&S f^{t}_{j,k}=
\sum\limits_{(\epsilon,s,u)\in G_{t,n}}S^{t}_{j} f^{\epsilon}_{j+s,2^{s}k+u}
\Phi^{\epsilon}_{j+s,2^{s}k+u}(x),\\
&P^{t}_{j}g^{t}_{j,k}
\end{array}\right.$$
indicate the micro-local characters in both the
frequency structure and the local structure.
\end{remark}

In the rest of this paper, these three kinds of quantities
obtained will be repeatedly used in the following sections.
Micro-local quantities result in the global information of functions in
$P^{\alpha}(\mathbb{R}^{n})$. In Section 4, this idea
will be used to get the wavelet characterization of
$P^{\alpha}(\mathbb{R}^{n})$ by a group of $L^1$ functions defined by the absolute values of
wavelet coefficients. Such wavelet characterization does not involve the action of a group of Borel measures.

\section{Wavelet characterization of  $P^{\alpha}(\mathbb{R}^{n})$}
In recent years, as the predual of
$Q_{\alpha}(\mathbb{R}^{n})$, $P^{\alpha}(\mathbb{R}^{n})$ was studied by many authors.
 We refer the reader to \cite{CY}, \cite{DX}, \cite{PY},
\cite{Yang}, \cite{YZ} and \cite{YSY} for details.
In this section, we give a new wavelet characterization of
$P^{\alpha}(\mathbb{R}^{n})$ by using the micro-local results obtained in
Section 3.

In the famous book \cite{Me}, Y. Meyer proved
that the Hardy space $H^{1}(\mathbb{R}^{n})$ can be characterized
by some $L^{1}$ function defined by wavelet coefficients
without using a family of Borel measures. Based on the micro-local results in Section 3,
 we give a similar result. We show that each element in
$P^{\alpha}(\mathbb{R}^{n})$ can be characterized by a group of
$L^{1}$ functions $P_{s,t,N}g(x)$. This result will be used to get
a characterization of $P^{\alpha}(\mathbb{R}^{n})$ associated with Riesz transformations in Section 5.

For $s\in \mathbb{Z}$ and $N\in \mathbb{N}$, let
$$\Omega^{s,N}=\Big\{Q\in \Omega: 2^{-sn}\leq |Q|\leq
2^{(N-s)n}\Big\}.$$
For $0\leq t\leq N, m\in\mathbb{Z}^{n}, Q=
Q_{s-N,m}$, define
$$\Omega_{s,t,Q}=\Omega^{t,N}_{s,m}=\Big\{Q'\in
\Omega: 2^{-sn}\leq |Q'|\leq 2^{(t-s)n}, Q'\subset
Q_{s-N,m}\Big\}.$$ We can see that
$\Omega^{s,N}=\bigcup\limits_{m\in \mathbb{Z}^{n}}
\Omega^{N,N}_{s,m}$. For any $s\in\mathbb{Z}$ and $N\in\mathbb{N}$,
we define
\begin{equation}\label{eq-g{s,N}}
g^{N}_{s-N,m}(x) = \sum\limits_{Q_{j,k}\in \Omega^{N,N}_{s,m}}
g^{\epsilon}_{j,k} \Phi^{\epsilon}_{j,k}(x)\end{equation}
and
\begin{equation}\label{eq2-g{s,N}}
g_{s,N}(x)=\sum\limits_{m\in\mathbb{Z}^{n}}
g^{N}_{s-N,m}(x).\end{equation}
 For $g(x)=\sum\limits_{(\epsilon,j,k)\in \Lambda_n} g^{\epsilon}_{j,k}
\Phi^{\epsilon}_{j,k}(x)\in P^{\alpha}(\mathbb{R}^{n})$, let $\Lambda^g_n=\Big\{
(\epsilon,j,k)\in \Lambda_n: g^{\epsilon}_{j,k}\neq 0\Big\}$.

Let $\{f^{\epsilon,g}_{j,k}\}$ be the sequences such that
$$f^{\epsilon,g}_{j,k}=\left\{
\begin{array}{cc}
|f^{\epsilon,g}_{j,k}||g^{\epsilon}_{j,k}|^{-1}
\overline{g^{\epsilon}_{j,k}}, & (\epsilon,j,k)\in \Lambda^g_n; \\
0, & (\epsilon,j,k)\notin \Lambda^g_n.
\end{array}\right.$$
We denote by $Q^{0,g}_{\alpha}$ the set
$$\Big\{f:\ f(x)= \sum\limits_{(\epsilon,j,k)\in \Lambda_n}f^{\epsilon, g}_{j,k}
\Phi^{\epsilon}_{j,k}(x) \text{ and }  \|f\|_{Q^0_{\alpha}}\leq 1\Big\}.$$

 By (\ref{3.1.1}), we have
\begin{equation}\label {eq:g}
\sup\limits_{\|f\|_{Q^0_{\alpha}}\leq 1} \tau_{f,g}=
\sup\limits_{f\in Q^{0,g}_{\alpha} } \tau_{f,g}.
\end{equation}

 We prove first an approximation lemma on $P^{\alpha}(\mathbb{R}^n)$.
\begin{lemma}\label{appr}
For $g(x) = \sum\limits_{(\epsilon,j,k)\in \Lambda_n}
g^{\epsilon}_{j,k} \Phi^{\epsilon}_{j,k}(x)\in P^{\alpha}(\mathbb{R}^n)$,
let
$$\tilde{g}_{s,N}(x)=\sum\limits_{|m|\leq 2^n}
g^{N}_{s-N,m}(x).$$ For arbitrary $\delta>0$, there exist $s$ and $N$
such that $\|g -\tilde{g}_{s,N} \|_{P^{\alpha}}\leq \delta$.
\end{lemma}
\begin{proof}
For any $0<\delta < \frac{\|g \|_{P^{\alpha}}}{8}$,
according to $(\alpha,2)-$wavelet atom decomposition, there exists
$\{\lambda_{u}\}_{u\in \mathbb{N_+}}\in l^1$ and a group of $(\alpha,2)-$wavelet atom $a_u(x)$
such that $g=\sum\limits_{u\in \mathbb{N_+}} \lambda_u a_u(x)$ and
$$\left|\sum\limits_{u\in \mathbb{N}} |\lambda_u|- \|g \|_{P^{\alpha}}\right|\leq \frac{\delta}{8}.$$
Further there exists integer $N_{\delta}>0$ such that
\begin{equation}\label{lambda}
\sum\limits_{u>N_{\delta}} |\lambda_u|\leq \frac{\delta}{8}.
\end{equation}

Now, for $u=1,\cdots, N_{\delta}$, we consider atoms
$$a_{u}(x)=
\sum\limits_{(\epsilon,j,k)\in \Lambda_{n}, Q_{j,k}\subset Q_u} a^{\epsilon,u}_{j,k}
\Phi^{\epsilon}_{j,k}(x).$$
 Since
$$\Big(\sum\limits_{(\epsilon,j,k)\in \Lambda_{n},  Q_{j,k}\subset Q_u} 2^{-2j\alpha} |a^{\epsilon,u}_{j,k}|^{2}\Big)
^{\frac{1}{2}} \leq |Q_u|^{\frac{\alpha}{n}-\frac{1}{2}},$$
then there exists integer $\tilde{N}_{\delta}>0$ such that
\begin{equation}\label {fin}
\Big(\sum\limits_{(\epsilon,j,k)\in \Lambda_{n},  Q_{j,k}\subset Q_u, j>\tilde{N}_{\delta}}
2^{-2j\alpha} |a^{\epsilon,u}_{j,k}|^{2}\Big)
^{\frac{1}{2}} \leq \frac{\delta }{16  \|g \|_{P^{\alpha}}} |Q_u|^{\frac{\alpha}{n}-\frac{1}{2}}.
\end{equation}
Since $1\leq u\leq N_{\delta}$, there exists an integer $j_{\delta}\in \mathbb{Z}$ such that
\begin{equation}\label {union}
\bigcup \limits_{1\leq u\leq N_{\delta}} Q_u \subset  \bigcup \limits_{|m|\leq 2^n}
Q_{j_{\delta},m}.
\end{equation}

For $u=1,\cdots, N_{\delta}$, let $b_{u}(x)=
\sum\limits_{(\epsilon,j,k)\in \Lambda_{n}, Q_{j,k}\subset Q_u,
j\leq \tilde{N}_{\delta}} a^{\epsilon,u}_{j,k}
\Phi^{\epsilon}_{j,k}(x)$.
According to (\ref{lambda}) and (\ref{fin}), we know that
\begin{equation}\label{12}
\begin{array}{cl}
&\Big\|\sum\limits_{u> N_{\delta}} \lambda_u a_u \Big\|_{P^{\alpha}}\leq \frac{\delta}{8};\\
&\Big\|\sum\limits_{1\leq u\leq N_{\delta}} \lambda_u (a_u-b_u)\Big\|_{P^{\alpha}}
\leq \sum\limits_{1\leq u\leq N_{\delta}} |\lambda_u| \frac{\delta }{16  \|g \|_{P^{\alpha}}} \leq \frac{\delta}{8}.
\end{array}\end{equation}
Let $$g_{\delta}(x)= \sum\limits_{u> N_{\delta}} \lambda_u a_u + \sum\limits_{1\leq u\leq N_{\delta}} \lambda_u (a_u-b_u)
= \sum\limits_{(\epsilon,j,k)\in \Lambda_n} g^{\epsilon,\delta}_{j,k} \Phi^{\epsilon}_{j,k}(x).$$
Then $\|g_{\delta}(x)\|_{P^{\alpha}}\leq \frac{\delta}{4}$.
Let $$g_{1,\delta}(x) = \sum\limits_{(\epsilon,j,k)\in \Lambda_n, j\geq N_{\delta},
Q_{j,k}\subset  \bigcup \limits_{|m|\leq 2^n} Q_{j_{\delta},m} }
g^{\epsilon,\delta}_{j,k} \Phi^{\epsilon}_{j,k}(x)$$ and
$g_{2,\delta}(x)=g_{\delta}(x)- g_{1,\delta}(x).$ According to (\ref{eq:g}), we have
$\|g_{1,\delta}(x)\|_{P^{\alpha}}\leq \frac{\delta}{4}$ and $\|g_{2,\delta}(x)\|_{P^{\alpha}}\leq \frac{\delta}{4}$.

Take $s=\tilde{N}_{\delta}$ and $N=s-j_{\delta}$ .
Let $\tilde{g}_{s,N}(x)= g_{1,\delta}(x)+\sum\limits_{1\leq u\leq N_{\delta}} \lambda_u b_u$.
According to the above construction process, $\tilde{g}_{s,N}(x)$ satisfies the condition of Lemma \ref{appr}.
\end{proof}
 Given $0\leq t\leq N$, $m\in\mathbb{Z}^{n}$ and $Q=Q_{s-N,m}$. If $t=0$, we denote
$$g^{\epsilon,s,t,N}_{j, k}=\left\{
\begin{array}{cc}
0, & j>s; \\
g^{\epsilon}_{j,k}, & j= s .
\end{array}\right.$$
For $t\geq 1$ and for $Q^{t}_{j} g^{\epsilon}_{j,k}$ defined in Theorem \ref{th4}, we denote
$$g^{\epsilon,s,t,N}_{j, k}=\left\{
\begin{array}{cc}
0, &j>s-t;\\
Q^{t}_{j} g^{\epsilon}_{j,k}, & j=s-t;\\
g^{\epsilon}_{j,k}, & j< s-t.
\end{array}\right.$$
Denote $g_{s,t,N}(x)=\sum\limits_{\epsilon,j,k}
g^{\epsilon,s,t,N}_{j,k} \Phi^{\epsilon}_{j,k}(x)$. We define
$$P_{s,t,N}g(x) =
\Big(\sum \limits_{\epsilon,Q_{j,k}\in \Omega^{s,N}, j\leq s-t}
2^{nj} |g^{\epsilon,s,t,N}_{j,k}|^{2} \chi
(2^{j}x-k)\Big)^{\frac{1}{2}},$$
$$Q_{s,t,N}g= \Big\|
\Big(\sum \limits_{\epsilon,Q_{j,k}\in \Omega^{s,N},j= s-t} 2^{jn}
|g^{\epsilon,s,t,N}_{j,k}|^{2} \chi
(2^{j}\cdot-k)\Big)^{\frac{1}{2}}\Big\|_{L^{1}}.$$
For $t=N$, we have \begin{equation}\label{eq:QP}
Q_{s,N,N}g^{N}_{s-N,m}=\|P_{s,N,N}g^{N}_{s-N,m}\|_{L^{1}}.\end{equation}

Now  we prove a wavelet characterization without involving Borel measures.
\begin{theorem}\label{cha}
If $0<\alpha<\frac{n}{2}$, then
$$P^{\alpha}(\mathbb{R}^{n})= \Big\{ g: \sup\limits_{s\in \mathbb{Z},N\in
\mathbb{N}} \min\limits_{0\leq t\leq N} \|P_{s,t,N}g\|_{L^{1}}
<\infty\Big\}.$$
\end{theorem}

\begin{proof}

%Applying (ii)   of Proposition \ref{pro:1}, we know that functions
%$$\sum\limits_{|m|\leq 2^n}\sum\limits_{Q_{j,k}\in \Omega^{N,N}_{s,m}}
%f^{\epsilon}_{j,k} \Phi^{\epsilon}_{j,k}(x)$$
%in $Q_{\alpha}$ are dense in $Q^0_{\alpha}$.

%Since $g(x)=\big [ g(x)- g_{s,N}(x)\big ] + g_{s,N}(x)$ and $g_{s,N}(x)=\sum\limits_{|m|\leq 2^n}
%g^{N}_{s-N,m}(x)+ \sum\limits_{|m|> 2^n}
% g^{N}_{s-N,m}(x)$,
According to (\ref{eq:g}) and Lemma \ref{appr},
$\forall \delta>0$, there exists $\tau_{\delta}>0$ such that for
$s>\tau_{\delta}, N\geq 2s$, we have
\begin{equation}\label{4.0.1}
\|g_{s,N}-g\|_{P^{\alpha}}+
\sum\limits_{|m|> 2^{n}} \| g^{N}_{s-N,m}\|_{P^{\alpha}}\leq \delta
\end{equation}
and
\begin{equation}
\label{4.0.2}
8^{-n} \max\limits_{|m|\leq 2^{n}} \|
g^{N}_{s-N,m}\|_{P^{\alpha}}- \delta \leq
\|g_{s,N}\|_{P^{\alpha}} \leq \sum\limits_{|m|\leq 2^{n}} \|
g^{N}_{s-N,m}\|_{P^{\alpha}}+ \delta,
\end{equation}
where $g_{s,N}$ and $g^{N}_{s-N,m}$ are defined by (\ref{eq-g{s,N}}) and (\ref{eq2-g{s,N}}).

By (\ref{eq:QP}) and Theorem \ref{th4}, we have
\begin{equation}\label{4.0.3}
\|g^{N}_{s-N,m}\|_{P^{\alpha}}
= Q_{s,N,N}g^{N}_{s-N,m}=\|P_{s,N,N}g^{N}_{s-N,m}\|_{L^{1}}.
\end{equation}
Furthermore, we have
\begin{equation}\label{4.0.4}
\|g_{s,t,N}\|_{P^{\alpha}} \leq \|g_{s,t,N}\|_{H^{1}}=\|P_{s,t,N}g\|_{L^{1}}.
\end{equation}
According to (\ref{4.0.1})-(\ref{4.0.4}),  the proof of Theorem \ref{cha} is completed.
\end{proof}

\section{Fefferman-Stein type decomposition of $Q$-spaces}
In this section,  by the wavelet
characterization obtained in  Theorem \ref{cha}, we give a
Fefferman-Stein type decomposition of $Q_{\alpha}(\mathbb{R}^{n})$.

In \cite{EJPX},  the authors proved that
$Q_{\alpha}(\mathbb{R}^{n})\subset BMO(\mathbb{R}^{n})$.
That is to say, the Hardy space $H^{1}(\mathbb{R}^{n})$ is contained in
$P^{\alpha}(\mathbb{R}^{n})$.
\begin{proposition}\label{php}
If $0<\alpha<\frac{n}{2}$, then $H^{1}(\mathbb{R}^{n})\subset
P^{\alpha}(\mathbb{R}^{n})$.
\end{proposition}

Further, we point out, for Theorem \ref{th:neq}, we need some special properties
of Daubechies wavelets. Except for Theorem \ref{th:neq},
we use the classical Meyer wavelets $\Phi^{\epsilon}(x)$
throughout Section 5 and Section 6. The support of the Fourier transform of
the classical Meyer wavelet in \cite{Me} satisfies the following conditions
\begin{equation}\label{wavelet condition}
\left\{
\begin{array}{cc}
&\text{ supp }\widehat{\Phi^{0}}(\xi)\subset [-\frac{4\pi}{3}, \frac{4\pi}{3}]; \\
&\text{ supp }\widehat{\Phi^{1}}(\xi)\subset [-\frac{8\pi}{3}, \frac{8\pi}{3}]
\backslash (-\frac{2\pi}{3}, \frac{2\pi}{3}).
\end{array}\right.\end{equation}
For tensor product  Meyer wavelet satisfying (\ref{wavelet condition}),
$\forall (\epsilon,j,k)$, $(\epsilon',j',k')$ $\in \Lambda_n$
and $|j-j'|\geq 2$, we have
\begin{equation}\label{5.0}
\langle R_i\Phi^{\epsilon}_{j,k},\
\Phi^{\epsilon'}_{j',k'}\rangle=0, \forall i=1,\cdots,n.
\end{equation}

\subsection{Adapted $L^{1}$ and $L^{\infty}$ spaces}
For $g(x)=\sum\limits_{(\epsilon,j,k)\in \Lambda_{n}}g^{\epsilon}_{j,k}
\Phi^{\epsilon}_{j,k}(x)$ and $j\in \mathbb{Z}$, denote
\begin{equation}\label{Qj}
Q_{j}g(x)= \sum\limits_{\epsilon\in E_n,k\in \mathbb{Z}^n}
g^{\epsilon}_{j,k}\Phi^{\epsilon}_{j,k}(x).\end{equation}
$\forall s\in \mathbb{Z}, N\in \mathbb{N}$, we set
\begin{equation}\label{sN}
P_{s,N}g(x)=\sum\limits_{\epsilon,s-N\leq j\leq s,k}
g^{\epsilon}_{j,k}\Phi^{\epsilon}_{j,k}(x).
\end{equation}
For all integers $t$ from 0 to $
N$, denote
$$T^{1}_{s,t,N}g(x)= \sum\limits_{\epsilon,s-t\leq j\leq
s,k} g^{\epsilon}_{j,k} \Phi^{\epsilon}_{j,k}(x)$$
 and
$$T^{2}_{s,t,N}g(x)= \sum\limits_{\epsilon,s-N\leq j< s-t,k}
g^{\epsilon}_{j,k} \Phi^{\epsilon}_{j,k}(x).$$
By Theorem \ref{cha}, we introduce a space $\tilde{P}^{\alpha}(\mathbb{R}^{n})$.
\begin{definition}Given $\alpha\in[0,\frac{n}{2})$.
$g(x)\in \tilde{P}^{\alpha}(\mathbb{R}^{n})$ if
$$\sup\limits_{s\in \mathbb{Z}, N\in
\mathbb{N}} \inf \limits_{0\leq t\leq N}
(\|T^{1}_{s,t,N}g\|_{P^{\alpha}}+\|T^{2}_{s,t,N}g\|_{H^{1}})
<\infty.$$
\end{definition}
This space is not really new. In fact,
\begin{theorem} \label{new}
(i) If $\alpha=0$, then $P^{0}(\mathbb{R}^{n})=\tilde{P}^{0}(\mathbb{R}^{n})=H^{1}(\mathbb{R}^{n})$.

(ii) If $0<\alpha<\frac{n}{2}$, then $P^{\alpha}(\mathbb{R}^{n})=\tilde{P}^{\alpha}(\mathbb{R}^{n})$.
\end{theorem}
\begin{proof} $P^{0}(\mathbb{R}^{n})=H^{1}(\mathbb{R}^{n})$ is known, so (i) is evident. Now we consider the cases
$0<\alpha<\frac{n}{2}$. If $g(x)\in P^{\alpha}(\mathbb{R}^{n})$, then
$\|P_{s,N}g\|_{P^{\alpha}}\leq \|g\|_{P^{\alpha}}$. Further
$$\inf \limits_{0\leq t\leq N}
(\|T^{1}_{s,t,N}g\|_{P^{\alpha}}+\|T^{2}_{s,t,N}g\|_{H^{1}})\leq
\|T^{1}_{s,N,N}g\|_{P^{\alpha}}=\|P_{s,N}g\|_{P^{\alpha}}.$$
Hence
$$\sup\limits_{s\in \mathbb{Z}, N\in \mathbb{N}} \inf
\limits_{0\leq t\leq N}
(\|T^{1}_{s,t,N}g\|_{P^{\alpha}}+\|T^{2}_{s,t,N}g\|_{H^{1}})
\leq \|g\|_{P^{\alpha}}.$$

Conversely, we have
\begin{eqnarray*}
\|P_{s,N}g\|_{P^{\alpha}}&\leq&
\|T^{1}_{s,t,N}g\|_{P^{\alpha}}+\|T^{2}_{s,t,N}g\|_{P^{\alpha}}\\
&\leq&
\|T^{1}_{s,t,N}g\|_{P^{\alpha}}+\|T^{2}_{s,t,N}g\|_{H^{1}}.\end{eqnarray*}
Hence
$$\|P_{s,N}g\|_{P^{\alpha}}\leq \inf \limits_{0\leq t\leq N}
(\|T^{1}_{s,t,N}g\|_{P^{\alpha}}+\|T^{2}_{s,t,N}g\|_{H^{1}}).$$
According to (\ref{4.0.1}), if $g(x)\in\tilde{P}^{\alpha}(\mathbb{R}^{n})$, then $g(x)\in
P^{\alpha}(\mathbb{R}^{n})$.
\end{proof}

Theorem \ref{new} tells us that the norm of a function $g(x)$ in
$P^{\alpha}(\mathbb{R}^{n})$ is equivalent to $$\sup\limits_{s\in \mathbb{Z}, N\in
\mathbb{N}} \inf \limits_{0\leq t\leq N}
(\|T^{1}_{s,t,N}g\|_{P^{\alpha}}
+\|T^{2}_{s,t,N}g\|_{H^{1}}).$$ That is to say, for
$0<\alpha<\frac{n}{2}$, the  higher
frequency part $T^{1}_{s,t,N}g(x)$ and the lower frequency part
$T^{2}_{s,t,N}g(x)$ make different contributions to the norm. Now we
use such property to construct $L^{1,\alpha}(\mathbb{R}^{n})$ and
$L^{\infty,\alpha}(\mathbb{R}^{n})$ which will be adapted to Fefferman-stein
decomposition.

Let $f(x)=\sum\limits_{(\epsilon,j,k)\in \Lambda_{n}}f^{\epsilon}_{j,k}
\Phi^{\epsilon}_{j,k}(x)$. For $s,t,N\in \mathbb{Z}$ and $ 0\leq t\leq N$,
we denote
\begin{eqnarray*}
P_{s,N}f(x)&=&\sum\limits_{\epsilon,s-N\leq j\leq s,k}
f^{\epsilon}_{j,k}\Phi^{\epsilon}_{j,k}(x),\\
S^{1}_{s,t,N}f(x)&=&
\sum\limits_{\epsilon,s-t\leq j\leq s,k} f^{\epsilon}_{j,k}
\Phi^{\epsilon}_{j,k}(x),\\
S^{2}_{s,t,N}f(x)&=&\sum\limits_{\epsilon,s-N\leq j< s-t,k} f^{\epsilon}_{j,k}
\Phi^{\epsilon}_{j,k}(x).\end{eqnarray*}

The spaces $L^{1,\alpha}(\mathbb{R}^{n})$ and $L^{\infty,\alpha}(\mathbb{R}^{n})$ are defined as follows.
\begin{definition}\label{L}
Given $f(x)=\sum\limits_{(\epsilon,j,k)\in \Lambda_{n}}f^{\epsilon}_{j,k}
\Phi^{\epsilon}_{j,k}(x)$ and $g(x)=\sum\limits_{(\epsilon,j,k)\in \Lambda_{n}}
g^{\epsilon}_{j,k} \Phi^{\epsilon}_{j,k}(x)$.

(i) We say that $g(x)\in L^{1,\alpha}(\mathbb{R}^{n})$ if
$$ \sup\limits_{s\in \mathbb{Z},N\in \mathbb{N}}
\min\limits_{0\leq t\leq N} \Big(\|T^{1}_{s,t,N}g\|_{ P^{\alpha}}
+\|T^{2}_{s,t,N}g\|_{ L^{1}}\Big) <\infty.$$

(ii) We say that $f(x) \in L^{\infty,\alpha}(\mathbb{R}^{n})$ if
$$\sup\limits_{s\in \mathbb{Z},N\in \mathbb{N}}
\sup\limits_{0\leq t\leq N} \Big(\|S^{1}_{s,t,N}f\|_{Q_{\alpha}}
+\|S^{2}_{s,t,N}f\|_{ L^{\infty}} \Big)<\infty.$$
\end{definition}

By Proposition \ref{php} and Theorem \ref{new}, we have
\begin{theorem}\label{subspace}
Given $0\leq \alpha<\frac{n}{2}$.
\begin{itemize}
\item[(i)] $P^{\alpha}(\mathbb{R}^{n})\subset L^{1,\alpha}(\mathbb{R}^{n})$;
\item[(ii)] $L^{\infty,\alpha}(\mathbb{R}^{n})=
Q_{\alpha}(\mathbb{R}^{n})\bigcap L^{\infty}(\mathbb{R}^{n})$;

\item[(iii)] $(L^{1,\alpha}(\mathbb{R}^{n}))'= L^{\infty,\alpha}(\mathbb{R}^{n})$.
\end{itemize}
\end{theorem}

\begin{remark}
For the case $\alpha=0$, we have:
\begin{itemize}
\item[(i)] $P^{0}(\mathbb{R}^{n})=H^{1}(\mathbb{R}^{n})$ and $Q_{0}(\mathbb{R}^{n})=BMO(\mathbb{R}^{n})$;
\item[(ii)] $L^{1,0}(\mathbb{R}^{n})=L^{1}(\mathbb{R}^{n})$ and $L^{\infty,0}(\mathbb{R}^{n})=L^{\infty}(\mathbb{R}^{n})$.
\end{itemize}
\end{remark}

Now, we use  Daubechies  wavelets to prove that $L^{\infty,\alpha}(\mathbb{R}^{n})
\varsubsetneq  Q_{\alpha}(\mathbb{R}^{n})$. We know that there exist some integer $M$
and Daubechies scale function $\Phi^{0}(x)\in C^{n+2}_{0} ([-2^{M}, 2^{M}]^{n})$
satisfying
\begin{equation}\label{5.1}
C_{D}=\int \frac{-y_{1}}{|y|^{n+1}}
\Phi^{0}(y-2^{M+1}e) dy<0,\end{equation}
where $e= (1,1,\cdots, 1)$.
\begin{theorem} \label{th:neq}
Let $\Phi(x)= \Phi^{0}(x-2^{M+1}e)$ and let
$f(x)$ be defined as
\begin{equation}\label{5.2}
f(x)=\sum\limits_{j\in 2\mathbb{N}} \Phi(2^{j}x).
\end{equation}
If $0\leq \alpha<\frac{n}{2}$, then
$f(x)\in L^{\infty,\alpha}(\mathbb{R}^{n})$ and $R_{1}f(x)\notin L^{\infty}(\mathbb{R}^{n}).$ That
is to say, $L^{\infty,\alpha}(\mathbb{R}^{n}) \varsubsetneq  Q_{\alpha}(\mathbb{R}^{n})$.
\end{theorem}

\begin{proof} For $j,j'\in 2\mathbb{N}, j\neq j'$,  the
supports of $\Phi(2^{j}x)$ and  $\Phi(2^{j'}x)$ are
disjoint. Hence the above $f(x)$ in (\ref{5.2}) belongs to $L^{\infty}(\mathbb{R}^{n})$.
The same reasoning gives, for any $j'\in \mathbb{N}$,
$$\sum\limits_{j\in \mathbb{N}, 2j> j'} \Phi(2^{2j}x)\in
L^{\infty}(\mathbb{R}^{n}).$$ Now we compute the wavelet coefficients of $f(x)$ in (\ref{5.2}).
For $(\epsilon',j',k')\in \Lambda_{n}$, let $f^{\epsilon'}_{j',k'}=
\langle f,\ \Phi^{\epsilon'}_{j',k'}\rangle$. We distinguish
two cases: $j'<0$ and $j'\geq 0$.

For $j'<0$, since the support of $f(x)$ is contained in $[-3 \cdot
2^{M}, 3 \cdot 2^{M}]^{n}$, we know that if $|k'|> 2^{2M+5}$, then
$f^{\epsilon'}_{j',k'}=0$. If $|k'|\leq 2^{2M+5}$, we have
$$|f^{\epsilon'}_{j',k'}| \leq C 2^{\frac{nj'}{2}} \int |f(x)| dx
\leq C 2^{\frac{nj'}{2}}.$$ For $j'\geq 0$, by orthogonality of the
wavelets, we have
$$f^{\epsilon'}_{j',k'}= \Big\langle f,\ \Phi^{\epsilon'}_{j',k'}\Big\rangle
= \Big\langle \sum\limits_{j\in \mathbb{N}, 2j> j'} \Phi(2^{2j}\cdot),\
\Phi^{\epsilon'}_{j',k'}\Big\rangle.$$
By the same reason,
for the case $j'\geq0$, we know that if $|k'|> 2^{2M+5}$, then
$f^{\epsilon'}_{j',k'}=0$. Since $\sum\limits_{j\in \mathbb{N}, 2j>
j'} \Phi(2^{2j}x)\in L^{\infty}$, if $|k'|\leq 2^{2M+5}$, we
have $$|f^{\epsilon'}_{j',k'}| \leq C\int
|\Phi^{\epsilon'}_{j',k'}(x)| dx \leq C 2^{-\frac{nj'}{2}}.$$ By the
above estimation of wavelet coefficients of $f(x)$ and by the wavelet
characterization of $Q$-spaces, we conclude that $f(x)\in Q_{\alpha}(\mathbb{R}^{n})$.
That is to say,
$$f(x)\in Q_{\alpha}(\mathbb{R}^{n})\cap L^{\infty}(\mathbb{R}^{n}).$$

Since $\Phi^{0}(x)\in C^{n+2}_{0}([-2^{M}, 2^{M}]^{n})$, we know that
$$\Phi(x)= \Phi^{0}(x-2^{M+1}e)\in C^{n+2}_{0}([2^{M}, 3 \cdot 2^{M}]^{n}).$$
Further,  if $|x|\leq 2^{M-1}$ and $y\in [2^{M},
3 \cdot 2^{M}]^{n} $, then $|x-y|> 2^{M-1}$. Hence $R_{1}\Phi(x)$ is
smooth in the ball $\{x:\ |x|\leq 2^{M-1}\}$.

Applying (\ref{5.1}), there exists a
positive $\delta>0$ such that for $|x|<\delta$,
$R_{1}\Phi(x)<\frac{C_{D}}{2}<0.$ That is to say, if
$2^{2j}|x|<\delta$, then $R_{1}\Phi(2^{2j}x)<\frac{C_{D}}{2}<0$.
Hence $$R_{1}f(x)\notin L^{\infty}(\mathbb{R}^{n}).$$ \end{proof}

\subsection{ Fefferman-Stein decomposition of $Q_{\alpha}(\mathbb{R}^{n})$}
C. Fefferman and E. M. Stein used Riesz transformations and $L^{1}$ norm to
characterize Hardy space $H^{1}(\mathbb{R}^{n})$, see \cite{FS}:
\begin{thm}\label{th:L^1}
$g(x)\in H^{1}(\mathbb{R}^{n})$ if and only if $$C_{Riesz}(g)= \|g(x)\|_{
L^{1}(\mathbb{R}^{n})}+ \sum\limits_{ i=1} ^{n} \|R_{i}g(x)\|_{ L^{1}(\mathbb{R}^{n})}<\infty.$$
\end{thm}

Theorem \ref{th:L^1} results in the solving of Fefferman-Stein decomposition of $BMO(\mathbb{R}^{n})$. The following theorem is a similar result for $P^{\alpha}(\mathbb{R}^{n})$ and it extends Theorem \ref{th:L^1}.
If $\alpha=0$, Theorem \ref {thL1} becomes Theorem \ref {th:L^1}, so we omit the proof of this case.
The proof for the cases $0< \alpha<\frac{n}{2}$ is very long.
So we  only state  this result  here and  postpone the proof to Section 6.
For $0\leq \alpha<\frac{n}{2}$ and function $g(x)$, denote
$$C_{\alpha,Riesz}(g)= \|g(x)\|_{L^{1,\alpha}}+ \sup\limits_{s\in \mathbb{Z}, N\in \mathbb{N}}
\min\limits_{0\leq t\leq N} \sum\limits_{i=1}^{n} \{\| T^{1}_{s,t,N}R_{i}g(x)\|_{P_{\alpha}}+
\| T^{2}_{s,t,N}R_{i}g(x)\|_{L^{1}}\}.$$
We have
\begin{theorem} \label{thL1}
If $0\leq \alpha<\frac{n}{2}$, then $g(x)\in
P^{\alpha}(\mathbb{R}^{n})$ if and only if
\begin{equation}\label{Riesz.c}
C_{\alpha,Riesz}(g)<\infty.\end{equation}
\end{theorem}

\begin{remark}\label{open}

For the cases $0< \alpha<\frac{n}{2}$, the authors \cite{EJPX} list several open problems about $Q_{\alpha}(\mathbb{R}^{n})$.
Their open problems attract a lot of attention and have been studied extensively  by many authors. See \cite{DX}, \cite{PY}, \cite{WX}, \cite{Yang},
\cite{YSY} and the references therein.  In the open problem 8.3, the authors \cite{EJPX} ask what would be the suitable subspaces of
$Q_{\alpha}(\mathbb{R}^{n})$ in which Fefferman-Stein decomposition is valid.

Unlike the situation of $BMO(\mathbb{R}^{n})$,  $L^{\infty}$ does not belong to $Q_{\alpha}(\mathbb{R}^{n})$. By duality,
$P^{\alpha}(\mathbb{R}^{n})$ is not a subspace of $L^{1}(\mathbb{R}^{n})$ either.
 If we want to solve the problem 8.3 by Fefferman-Stein's idea, we need to establish a relation between $P^{\alpha}(\mathbb{R}^{n})$ and the functions in $L^{1}(\mathbb{R}^{n})$. For this purpose, we should apply much more skills. For example, in Section 3, we  consider the micro-local property of $P_{\alpha}(\mathbb{R}^{n})$.
In Subsection 5.1, we construct the space $L^{1,\alpha}(\mathbb{R}^{n})$ by Meyer wavelets.

 \end{remark}

If Theorem \ref{thL1} holds, by Theorem \ref{th:neq},
we could obtain a Fefferman-Stein type decomposition of $Q_{\alpha}(\mathbb{R}^{n})$ using Fefferman-Stein's skill in \cite{FS}.
This result solves Problem \ref{p1} (Problem 8.3 in \cite{EJPX}).

\begin{theorem}\label{th2} If $0\leq \alpha<\frac{n}{2}$, then
$f(x)\in Q_{\alpha}(\mathbb{R}^{n})$ if and only if $f(x)=\sum\limits_{0\leq i\leq
n}R_{i}f_{i}(x)$, where $f_{i}(x)\in Q_{\alpha}(\mathbb{R}^{n})\bigcap L^{\infty}(\mathbb{R}^{n})$.
\end{theorem}

\begin{proof} By the continuity of the Calder\'on-Zygmund operators on $Q$-spaces, we know that if  $f_{i}(x)\in Q_{\alpha}(\mathbb{R}^{n})\bigcap
L^{\infty}(\mathbb{R}^{n})$, then
$$\sum\limits_{0\leq i\leq n}R_{i}f_{i}(x) \in
Q_{\alpha}(\mathbb{R}^{n}).$$

Now we prove the converse result. Let
$$B=\Big\{(g_{0},g_{1},\cdots,g_{n}): g_{i}\in L^{1,\alpha}(\mathbb{R}^{n}),
i=0,\cdots,n\Big\}.$$
The norm of $B$ is defined as $$\|(g_{0},g_{1},\cdots,g_{n})\|_{B}=
\sum\limits^{n}_{i=0}\|g_{i}\|_{L^{1,\alpha}}.$$
 We define
$$S=\Big\{(g_{0},g_{1},\cdots,g_{n})\in B: g_{i}=R_{i}g_{0}, i=0, 1,\cdots, n\Big\}.$$
 $S$ is a
closed subset of $B$. By theorem \ref{thL1}, $g_{0}\rightarrow
(g_{0},R_{1}g_{0},\cdots,R_{n}g_{0})$ define a norm preserving map
from $P^{\alpha}(\mathbb{R}^{n})$ to $S$. Hence the set of continuous linear functional on
$P^{\alpha}(\mathbb{R}^{n})$ is equivalent to the set of continuous linear functionals on
$S$.  By Hahn-Banach theorem,
which can extend to a continuous linear functionals on $B$ preserving the same norm.
We know that the dual space of $L^{1,\alpha}(\mathbb{R}^{n})\oplus \cdots\oplus
L^{1,\alpha}(\mathbb{R}^{n})$ is $L^{\infty,\alpha}(\mathbb{R}^{n})\oplus \cdots\oplus
L^{\infty,\alpha}(\mathbb{R}^{n})$.

$\forall f\in Q_{\alpha}(\mathbb{R}^{n})$, $f$ define a continuous linear functional
$l$ on  $P^{\alpha}(\mathbb{R}^{n})$ and also on $S$. Hence there exist
$\tilde{f}_{i}\in L^{\infty,\alpha}(\mathbb{R}^{n}), i=0,1,\cdots,n,$ such that for any $g_{0}\in P^{\alpha}(\mathbb{R}^{n})$,
\begin{eqnarray*}
 l(f)& = &\int_{\mathbb{R}^{n}}f(x) g_{0}(x) dx \\
 &=&\int_{\mathbb{R}^{n}}\tilde{f}_{0}(x) g_{0}(x) dx
+\sum\limits^{n}_{i=1} \int_{\mathbb{R}^{n}}
\tilde{f}_{i}(x) R_{i}g_{0}(x) dx\\
&=&\int_{\mathbb{R}^{n}}\tilde{f}_{0}(x) g_{0}(x) dx
-\sum\limits^{n}_{i=1} \int_{\mathbb{R}^{n}} R_{i}(\tilde{f}_{i})(x)
g_{0}(x) dx.
\end{eqnarray*}

Hence $f(x)= \tilde{f}_{0}(x) -\sum\limits^{n}_{i=1}
R_{i}(\tilde{f}_{i})(x)$.
\end{proof}

Triebel-Lizorkin spaces $\dot{F}^{0,q}_{\infty}$ are introduced in \cite{Tr},  Besov-Morrey spaces and Triebel-Lizorkin-Morrey spaces are introduced in
\cite{YSY}. These spaces play an important role in Harmonic analysis and non-linear problem, \cite{LXY} etc. Since Fefferman-Stein decomposition of BMO plays an important role in harmonic analysis, we propose the following open problems:
\begin{remark}
(1) In dimension $n$, how to give Fefferman-Stein decomposition for Triebel-Lizorkin spaces $\dot{F}^{0,q}_{\infty}$? 

(2) More generally, for other Besov-Morrey spaces or Triebel-Lizorkin-Morrey spaces, whether there is also Fefferman-Stein decomposition?
\end{remark}

\section{The proof of Theorem \ref{thL1}}
In \cite{FS}, C. Fefferman and E. M. Stein  used the Riesz transformations to
characterize Hardy space $H^{1}(\mathbb{R}^{n})$ in terms of the $L^{1}$ norm.
In this section, we prove the similar result
Theorem \ref{thL1} for the cases $0<\alpha<\frac{n}{2}$.

First, we prove that
\begin{equation}\label{sufficient}
g(x)\in
P^{\alpha}(\mathbb{R}^{n})\Longrightarrow g(x) \mbox{ satisfies the condition (\ref {Riesz.c}).}
% R_{i}g(x) \in L^{1,\alpha}(\mathbb{R}^{n}),\ i=0, 1,\cdots,n.
\end{equation}

By (ii) of Proposition \ref{pro:CZ}, Riesz transforms are bounded on $P^{\alpha}(\mathbb{R}^{n})$. The fact $g(x)\in
P^{\alpha}(\mathbb{R}^{n})$ implies that
$$R_{i}g(x) \in
P^{\alpha}(\mathbb{R}^{n}),\ i=1,\cdots,n.$$
 By (i) of Theorem \ref{subspace},
we have
$$ g(x) \mbox{ satisfies the condition (\ref {Riesz.c}).}
%R_{i}g(x)\in L^{1,\alpha}(\mathbb{R}^{n}),\ i=0,1, \cdots, n.
$$

The proof of the converse of (\ref{sufficient}) is cumbersome. We will prove it in the subsection 6.2. Here, we prove first a lemma in the next subsection 6.1.

\subsection{A lemma}
We first prove the following lemma.
\begin{lemma}\label{le:1t}
For $g(x)=\sum\limits_{(\epsilon,j,k)\in \Lambda_{n}}
g^{\epsilon}_{j,k} \Phi^{\epsilon}_{j,k}(x)$ and arbitrary $j\in \mathbb{Z}$, denote
$g_{j}(x)=\sum\limits_{\epsilon\in E_n,k\in \mathbb{Z}^{n}}
g^{\epsilon}_{j,k}\Phi^{\epsilon}_{j,k}(x)$
and denote $\tilde{g}_j(x)= \sum\limits_{j'\leq j} g_{j'}(x)$.
For $0<\alpha<\frac{n}{2}$, we have
\begin{itemize}
\item[(i)] $\|g_{j}\|_{H^{1}}\leq C\|g\|_{L^{1}}$.

\item[(ii)] $\max\Big\{\|\tilde{g}_{j}\|_{P^{\alpha}},\ \|g-\tilde{g}_{j}\|_{P^{\alpha}}\Big\}
\leq \|g\|_{P^{\alpha}}\leq
\|\tilde{g}_{j}\|_{P^{\alpha}}+ \|g-\tilde{g}_{j}\|_{P^{\alpha}}$.

\item[(iii)]  $\|g_{j}\|_{P^{\alpha}}\leq C\|g\|_{L^{1,\alpha}}$.\end{itemize}
\end{lemma}

\begin{proof}
(i) By applying wavelet characterization of Hardy spaces and the orthogonality
properties of the Meyer wavelets, we have
\begin{eqnarray*}
\|g_{j}\|_{H^{1}}&\leq &C \Big\|\Big(\sum\limits_{\epsilon\in E_{n},k\in \mathbb{Z}^{n}} 2^{nj}|
\langle g_{j}, \Phi^{\epsilon}_{j,k}\rangle |^{2}\chi(2^{j}\cdot-k)\Big)^{\frac{1}{2}}\Big\|_{L^{1}}\\
&\leq & C \sum\limits_{\epsilon\in E_{n} } \Big\|\sum\limits_{k\in \mathbb{Z}^{n}} 2^{\frac{n}{2}j}|
\langle g, \Phi^{\epsilon}_{j,k}\rangle |\chi(2^{j}\cdot-k)\Big\|_{L^{1}}\\
&\leq &C\|g\|_{L^{1}}.
\end{eqnarray*}

(ii) $P^{\alpha}(\mathbb{R}^{n})$ is a Banach space, hence we have
$$\|g\|_{P^{\alpha}}\leq
\|\tilde{g}_{j}\|_{P^{\alpha}}+ \|g-\tilde{g}_{j}\|_{P^{\alpha}}.$$

To prove the first inequality of (ii), denote
$$G_{g}=\Big\{(\epsilon,j,k)\in \Lambda_{n}, g^{\epsilon}_{j, k}\neq 0\Big\}.$$
For
$f(x)= \sum\limits_{(\epsilon,j,k)\in \Lambda_{n}} f^{\epsilon}_{j,k}
\Phi^{\epsilon}_{j,k}(x)$, define
$$f^{\epsilon,g}_{j,k}=\left\{\begin{array}{cl}
|f^{\epsilon}_{j,k}|\cdot
|g^{\epsilon}_{j,k}|^{-1}  \overline{g^{\epsilon}_{j,k}},
&(\epsilon,j,k)\in G_{g};\\
0, & (\epsilon,j,k)\notin G_{g}.
\end{array}\right.$$
We denote by $F_{g}$ the set
$$\Big\{f :\ f(x)=
\sum\limits_{(\epsilon,j,k)\in G_{g}} f^{\epsilon,g}_{j,k}
\Phi^{\epsilon}_{j,k}(x) \text{ and } \|f\|_{Q^0_{\alpha}}\leq 1\Big\}.$$

Define
$$\underline{f}(x)=\sum\limits_{(\epsilon,j,k)\in \Lambda_{n}} |f^{\epsilon}_{j,k}|
\Phi^{\epsilon}_{j,k}(x)$$
and
$$\underline{g}(x)=\sum\limits_{(\epsilon,j,k)\in \Lambda_{n}} |g^{\epsilon}_{j,k}|
\Phi^{\epsilon}_{j,k}(x).$$

By (ii) of Proposition \ref{pro:1}, we have
\begin{equation}\label{eq:eq}
\sup\limits_{\|f\|_{Q^0_{\alpha}\leq 1}} \tau_{f,g }
=  \sup\limits_{f\in F_g} \tau_{f,g }
=  \sup\limits_{\underline{f}\in F_g } \tau_{\underline{f},\underline{g} }
=  \sup\limits_{ \|\underline{f} \|_{Q^0_{\alpha}\leq 1} } \tau_{\underline{f},\underline{g} }.
\end{equation}

Hence  we have
$$\max\Big\{\|\tilde{g}_{j}\|_{P^{\alpha}},\ \|g-\tilde{g}_{j}\|_{P^{\alpha}}\Big\}
\leq \|g\|_{P^{\alpha}}.$$

(iii) By the definition of the norm of $g(x)$ in $L^{1,\alpha}(\mathbb{R}^{n})$, for any
$s\in \mathbb{Z}$, $N\in \mathbb{N}$ and $s-N\leq j\leq s$, there exists $j_0$ such that $0\leq j_0\leq N$ and
$$\Big\|\sum\limits_{s-j_{0}\leq j'\leq s} g_{j'}\Big\|_{P^{\alpha}}
+\Big\|\sum\limits_{s-N\leq  j'< s-j_{0}} g_{j'}\Big\|_{L^{1}}\leq
\|g\|_{L^{1,\alpha}}.$$
If $j<s-j_{0}$, we apply  (i) of this lemma to get the desired assertion.
If $j\geq s-j_{0}$, we apply  (ii) of this lemma to get the desired assertion.
\end{proof}

\subsection{The proof of converse part }
For the proof of the converse  of (\ref{sufficient}),
it is sufficient to prove that
$\forall s_{1}\in \mathbb{Z}$, $N_{1}\geq 1$ and
$g_{s_{1},N_{1}}(x)=P_{s_{1},N_{1}}g(x)$  defined in
 (\ref{sN}), we have
\begin{equation}\label {eq:5.3}
\|g_{s_{1},N_{1}}\|_{ P^{\alpha}} \leq C C_{\alpha,Riesz}(g_{s_{1},N_{1}})
%\sum\limits_{i=0,\cdots,n} \|R_{i}g_{s_{1},N_{1}}\|_{ L^{1,\alpha}}
.\end{equation}
Owing to (\ref{5.0}), there exists
$\{g^{\epsilon,i}_{j,k}\}_{(\epsilon,j,k)\in \Lambda_n}$ such that for $ i=1,2,\cdots,n$,
\begin{equation}\label {eq:5.4}
R_{i}g_{s_{1},N_{1}}(x) =
\sum\limits_{(\epsilon,j,k)\in \Lambda_n, s_{1}-N_{1}-1\leq j
\leq s_{1}+1} g^{\epsilon,i}_{j,k}
\Phi^{\epsilon}_{j,k}(x).
\end{equation}

Due to (\ref{eq:5.4}), to estimate the $L^{1,\alpha}-$norm of
$$R_{i}g_{s_{1},N_{1}}(x), i=0,1,\cdots, n,$$
it is sufficient
to consider $s=s_{1}+1$ and $N=N_{1}+2$. For such $s$ and $N$,
there exist  $t_{s,N}^{0}$ and $t_{s,N}^{1}$ such that
\begin{equation}\label{eq:5.5}
\begin{array}{rl}
&\|T^{1}_{s,t_{s,N}^{0},N}g_{s_{1},N_{1}}\|_{ P^{\alpha}}
+\|T^{2}_{s,t_{s,N}^{0},N}g_{s_{1},N_{1}}\|_{ L^{1}}\\
=& \min\limits_{0\leq t\leq N}
\Big(\|T^{1}_{s,t,N}g_{s_{1},N_{1}}\|_{ P^{\alpha}}
+\|T^{2}_{s,t,N}g_{s_{1},N_{1}}\|_{ L^{1}}\Big);
\end{array}
\end{equation}
\begin{equation}\label{eq:5.6}
\begin{array}{rl}
&\sum\limits_{1\leq i\leq n}\{\|  T^{1}_{s,t_{s,N}^{1},N}R_{i}g_{s_{1},N_{1}}\|_{ P^{\alpha}}
+\| T^{2}_{s,t_{s,N}^{1},N}R_{i}g_{s_{1},N_{1}}\|_{ L^{1}}\}\\
=& \min\limits_{0\leq t\leq N}
\sum\limits_{1\leq i\leq n}\{\|R_{i} T^{1}_{s,t,N} g_{s_{1},N_{1}}\|_{ P^{\alpha}}
+\|R_{i} T^{2}_{s,t,N} g_{s_{1},N_{1}}\|_{ L^{1}}\}.
\end{array}
\end{equation}

We divide the proof into three cases.

{\bf Case 1: \ $t^{0}_{s,N}=t^{1}_{s,N}.$}
Let $Q_{j}$ be the projection operators defined by (\ref{Qj}).
We divide the function $g_{s_1,N_1}$ into two functions
$$g_{s_1,N_1}(x)=g^{1}_{s,N}(x)+g^{2}_{s,N}(x),$$
where
$$g^{1}_{s,N}(x)=\sum\limits_{j\geq
s-t^{0}_{s,N}} Q_{j}g_{s_{1},N_{1}}(x)$$
 and
$$g^{2}_{s,N}(x)=\sum\limits_{j< s-t^{0}_{s,N}}
Q_{j}g_{s_{1},N_{1}}(x).$$

By (\ref{eq:5.5}), we have $g^{2}_{s,N}(x)\in
L^{1}(\mathbb{R}^{n})$. By Lemma \ref{le:1t},
\begin{equation} \label{eq:5.7}
Q_{s-t^0_{s,N}-1}g_{s_1,N_1}(x)+ Q_{s-t^0_{s,N}-2}g_{s_1,N_1}(x)\in H^{1}(\mathbb{R}^n)
\end{equation}
and
\begin{equation}\label{eq:5.10}
g^{2}_{s,N}(x)-\Big(Q_{s-t^{0}_{s,N}-1}g_{s_{1},N_{1}}(x)
+Q_{s-t^{0}_{s,N}-2}g_{s_{1},N_{1}}(x)\Big)\in L^{1}(\mathbb{R}^{n}).\end{equation}

Further, for $i=1,\cdots,n$, we have
\begin{eqnarray*}
&& T^{2}_{s, t^{1}_{s,N},N}R_{i}g_{s_{1},N_{1}}(x)\\
&=&
T^{2}_{s, t^{1}_{s,N},N}R_{i}\Big[g^{2}_{s,N}(x)+Q_{s-t^{0}_{s,N}}g_{s_{1},N_{1}}(x)\Big]\\
&=&T^{2}_{s,t^{1}_{s,N},N}R_{i}\Big[g^{2}_{s,N}(x)-\Big(Q_{s-t^{0}_{s,N}-1}g_{s_{1},N_{1}}(x)
+Q_{s-t^{0}_{s,N}-2}g_{s_{1},N_{1}}(x)\Big)\\
&&+\Big(Q_{s-t^{0}_{s,N}-1}g_{s_{1},N_{1}}(x)
+Q_{s-t^{0}_{s,N}-2}g_{s_{1},N_{1}}(x)\Big)+Q_{s-t^{0}_{s,N}}g_{s_{1},N_{1}}(x)\Big]\\
&=&R_{i}\Big[g^{2}_{s,N}(x)-\Big(Q_{s-t^{0}_{s,N}-1}g_{s_{1},N_{1}}(x)
+Q_{s-t^{0}_{s,N}-2}g_{s_{1},N_{1}}(x)\Big)\Big]\\
&&+T^{2}_{s,t^{1}_{s,N},N}R_{i}\Big[Q_{s-t^{0}_{s,N}-1}g_{s_{1},N_{1}}(x)
+Q_{s-t^{0}_{s,N}-2}g_{s_{1},N_{1}}(x)\Big]\\
&&+T^{2}_{s,t^{1}_{s,N},N}R_{i}Q_{s-t^{0}_{s,N}}g_{s_{1},N_{1}}(x).
\end{eqnarray*}
Hence, by the equation (\ref {eq:5.7}), for $i=1,\cdots,n$,
\begin{eqnarray*}
II^i(x)&=:& R_{i}\Big[g^{2}_{s,N}(x)-\Big(Q_{s-t^{0}_{s,N}-1}g_{s_{1},N_{1}}(x)
+Q_{s-t^{0}_{s,N}-2}g_{s_{1},N_{1}}(x)\Big)\Big]\\
&&+ T^{2}_{s,
t^{1}_{s,N},N}R_{i}Q_{s-t^{0}_{s,N}}g_{s_{1},N_{1}}(x)\in
L^{1}(\mathbb{R}^{n}).\end{eqnarray*}

By equation (\ref{5.0}), there exists $\{\tau^{\epsilon,i}_{j,k}\}_{(\epsilon,j,k)\in \Lambda_n}$ such that
\begin{eqnarray*}
I^i(x)&=:&R_{i}\Big[g^{2}_{s,N}(x)-\Big(Q_{s-t^{0}_{s,N}-1}g_{s_{1},N_{1}}(x)+Q_{s-t^{0}_{s,N}-2}g_{s_{1},N_{1}}(x)\Big)\Big]\\
&=&\sum\limits_{(\epsilon,j,k)\in \Lambda_n, j\leq s-t^{0}_{s,N}-2}
\tau^{\epsilon,i}_{j,k}\Phi^{\epsilon}_{j,k}(x)
\end{eqnarray*} and
$$R_{i}Q_{s-t^{0}_{s,N}}g_{s_{1},N_{1}}(x)=
\sum\limits_{(\epsilon,j,k)\in \Lambda_n, s-t^{0}_{s,N}-1\leq j\leq s-t^{0}_{s,N}+1}
\tau^{\epsilon,i}_{j,k}\Phi^{\epsilon}_{j,k}(x).$$
For arbitrary $L^{\infty}$ function
$$h(x)= \sum\limits_{(\epsilon,j,k)\in \Lambda_n}
h^{\epsilon}_{j,k}\Phi^{\epsilon}_{j,k}(x)$$ and $j_0\in \mathbb{Z}$, denote the operator
$$P_{j_0}h(x)=\sum\limits_{k\in \mathbb{Z}^n}
\langle h(x), \Phi^0_{j_0,k}(x)\rangle \Phi^0_{j_0,k}(x).$$
We can see that $P_{j_0}h(x)\in L^{\infty}(\mathbb{R}^{n})$. In fact, by the fact
$$|\langle h(x), \Phi^0_{j_0,k}(x)\rangle|\leq C2^{-\frac{nj_0}{2}},$$
we can get
\begin{eqnarray*}
|P_{j_0}h(x)| &\leq& C\sum\limits_{k\in \mathbb{Z}^n}
2^{-\frac{nj_0}{2}}|\Phi^0_{j_0,k}(x)|\\
&\leq& C\sum\limits_{k\in \mathbb{Z}^n}
|\Phi^0(2^{j_0}x-k)|\\
&\leq& C.
\end{eqnarray*}

Let
$$h_0(x)=: P_{s-t^{0}_{s,N}-2}h(x)= \sum\limits_{(\epsilon,j,k)\in \Lambda_n, j\leq s-t^{0}_{s,N}-2}
h^{\epsilon}_{j,k}\Phi^{\epsilon}_{j,k}(x).$$
Hence $h_0(x)\in L^{\infty}(\mathbb{R}^{n})$.
Further, because $II^{i}(x)\in L^{1}(\mathbb{R}^{n})$,
\begin{eqnarray*}
|\langle I^{i}, h\rangle|&=&|\langle I^{i}, h_{0}\rangle|\\
&=&|\langle II^{i}, h_{0}\rangle|\\
&\leq&\|II^{i}\|_{L^{1}}\|h_{0}\|_{\infty}.
\end{eqnarray*}
The last estimate implies that for $
i=1, \cdots, n,$ the functions
\begin{eqnarray*}
I^{i}(x)&=&R_{i}\Big[g^{2}_{s,N}(x)-\Big(Q_{s-t^{0}_{s,N}-1}g_{s_{1},N_{1}}(x)\\
&&+Q_{s-t^{0}_{s,N}-2}g_{s_{1},N_{1}}(x)\Big)\Big]\in L^{1}(\mathbb{R}^{n}).
\end{eqnarray*}
This fact and (\ref{eq:5.10}) imply that
$$g^{2}_{s,N}(x)-\Big(Q_{s-t^{0}_{s,N}-1}g_{s_{1},N_{1}}(x)
+Q_{s-t^{0}_{s,N}-2}g_{s_{1},N_{1}}(x)\Big)\in H^{1}(\mathbb{R}^{n}).$$
By (\ref{eq:5.7}), we get
$$g^{2}_{s,N}(x)\in H^{1}(\mathbb{R}^{n}).$$
Further, we have $g^{1}_{s,N}(x)\in P^{\alpha}(\mathbb{R}^{n})$. Applying (\ref{eq:5.5}),
we get $$g_{s_{1},N_{1}}(x)\in P^{\alpha}(\mathbb{R}^{n}).$$

{\bf Case 2: $t^{0}_{s,N}>t^{1}_{s,N}.$} For this case, we  decompose $g_{s_{1},N_{1}}(x)$ into
three functions
$$g_{s_{1},N_{1}}(x)=g^{1}_{s,N}(x)+g^{2}_{s,N}(x)+g^{3}_{s,N}(x), $$
where
$$g^{1}_{s,N}(x)=\sum\limits_{j\geq
s-t^{1}_{s,N}} Q_{j}g_{s_{1},N_{1}}(x),$$
$$g^{2}_{s,N}(x)=\sum\limits_{s-t^{0}_{s,N}\leq j< s-t^{1}_{s,N}}
Q_{j}g_{s_{1},N_{1}}(x)$$ and $$g^{3}_{s,N}(x)=\sum\limits_{j<
s-t^{0}_{s,N}} Q_{j}g_{s_{1},N_{1}}(x).$$ We know that
\begin{eqnarray*}
&&T^{2}_{s, t^{1}_{s,N},N}R_{i}g_{s_{1},N_{1}}(x)\\
&=&T^{2}_{s,
t^{1}_{s,N},N}R_{i}\Big[g^{3}_{s,N}(x)+g^{2}_{s,N}(x)+Q_{s-t^{1}_{s,N}}g_{s_{1},N_{1}}(x)\Big].\end{eqnarray*}
Then $\forall h(x)= \sum\limits_{\epsilon, s-N\leq
j<s-t^{0}_{s,N},k} h^{\epsilon}_{j,k}\Phi^{\epsilon}_{j,k}(x)$ and
$\|h\|_{L^{\infty}}\leq 1$, we know that
\begin{equation}\label{eq:5.8}
\langle T^{2}_{s,
t^{1}_{s,N},N}R_{i}g_{s_{1},N_{1}},\ h\rangle = \langle
R_{i}g^{3}_{s,N},\ h\rangle.
\end{equation}
By (\ref{eq:5.5}), $g^{3}_{s,N}(x)\in L^{1}(\mathbb{R}^{n})$. This fact implies that
\begin{equation}\label{eq:5.9}
Q_{s-t^{0}_{s,N}-1}g_{s_{1},N_{1}}(x)+Q_{s-t^{0}_{s,N}-2}g_{s_{1},N_{1}}(x)\in H^{1}(\mathbb{R}^{n}).
\end{equation}
Owing to (\ref{eq:5.8}) and (\ref{eq:5.9}), for $i=0,\cdots,n$, we have
$$R_{i}\Big[g^{3}_{s,N}(x)-\Big(Q_{s-t^{0}_{s,N}-1}g_{s_{1},N_{1}}(x)
+Q_{s-t^{0}_{s,N}-2}g_{s_{1},N_{1}}(x)\Big)\Big]\in L^{1}(\mathbb{R}^{n}).$$
 Hence we obtain
$$g^{3}_{s,N}(x)-\Big(Q_{s-t^{0}_{s,N}-1}g_{s_{1},N_{1}}(x)
+Q_{s-t^{0}_{s,N}-2}g_{s_{1},N_{1}}(x)\Big)\in H^{1}(\mathbb{R}^{n}).$$
So we have $g^{3}_{s,N}(x)\in H^{1}(\mathbb{R}^{n})$. Since
$$g^{1}_{s,N}(x)+g^{2}_{s,N}(x) \in P^{\alpha}(\mathbb{R}^{n}),$$ we
have $g_{s_{1},N_{1}}(x)\in P^{\alpha}(\mathbb{R}^{n})$.

{\bf Case 3: $t^{0}_{s,N}<t^{1}_{s,N}.$}
We  decompose
$g_{s_{1},N_{1}}(x)$ into three functions
$$g_{s_{1},N_{1}}(x)=g^{1}_{s,N}(x)+g^{2}_{s,N}(x)+g^{3}_{s,N}(x),$$
where $$g^{1}_{s,N}(x)=\sum\limits_{j\geq s-t^{0}_{s,N}}
Q_{j}g_{s_{1},N_{1}}(x),$$
$$g^{2}_{s,N}(x)=\sum\limits_{s-t^{1}_{s,N}\leq j< s-t^{0}_{s,N}}
Q_{j}g_{s_{1},N_{1}}(x)$$
 and $$g^{3}_{s,N}(x)=\sum\limits_{j<
s-t^{1}_{s,N}} Q_{j}g_{s_{1},N_{1}}(x).$$

 For $i=1,\cdots,n$, we know
that
\begin{eqnarray*}
&&T^{2}_{s, t^{1}_{s,N},N}R_{i}g_{s_{1},N_{1}}(x)\\&=&
T^{2}_{s, t^{1}_{s,N},N}R_{i}\Big[g^{3}_{s,N}(x)+Q_{s-t^{1}_{s,N}}g_{s_{1},N_{1}}(x)\Big]\\
&=&R_{i}\Big[g^{3}_{s,N}(x)-\Big(Q_{s-t^{1}_{s,N}-1}g_{s_{1},N_{1}}(x)
+Q_{s-t^{1}_{s,N}-2}g_{s_{1},N_{1}}(x)\Big)\Big]\\
&+& T^{2}_{s, t^{1}_{s,N},N}R_{i}\Big[Q_{s-t^{1}_{s,N}-1}g_{s_{1},N_{1}}(x)
+Q_{s-t^{1}_{s,N}-2}g_{s_{1},N_{1}}(x)\Big]\\
&+& T^{2}_{s,
t^{1}_{s,N},N}R_{i}Q_{s-t^{1}_{s,N}}g_{s_{1},N_{1}}(x).
\end{eqnarray*}
Define $h_{i}(x)$, $i=1,2,3,4,$ as
\begin{eqnarray*}
h_{1}(x)&=&
\sum\limits_{\epsilon, s-N\leq j<s-t^{1}_{s,N}-2,k}
h^{\epsilon,1}_{j,k}\Phi^{\epsilon}_{j,k}(x),\\
h_{2}(x)&=&
\sum\limits_{\epsilon, j=s-t^{1}_{s,N}-2,k}
h^{\epsilon,2}_{j,k}\Phi^{\epsilon}_{j,k}(x),\\
h_{3}(x)&=&
\sum\limits_{\epsilon, j=s-t^{1}_{s,N}-1,k}
h^{\epsilon, 3}_{j,k}\Phi^{\epsilon}_{j,k}(x),\\
h_{4}(x)&=&
\sum\limits_{\epsilon, j=s-t^{1}_{s,N},k}
h^{\epsilon, 4}_{j,k}\Phi^{\epsilon}_{j,k}(x),
\end{eqnarray*}
where the sequences $\{h^{\epsilon,i}_{j,k}\}, i=1, 2, 3, 4$, are four arbitrary sequences
satisfying  the condition $\|h_{i}\|_{L^{\infty}}\leq 1$.

 We consider
 $$\int T^{2}_{s,
t^{1}_{s,N},N}R_{i}g_{s_{1},N_{1}}(x) h_{i}(x)dx.$$
 By  (\ref{eq:5.6}) and the definition of
$t^{1}_{s,N}$, we have
\begin{equation}\label{eq:23}
g^{3}_{s,N}(x)+g^{2}_{s,N}(x)\in L^{1}(\mathbb{R}^{n}).
\end{equation}
Hence
\begin{equation}\label{eq:5.10}
g^{2}_{s,N}(x)-\Big(Q_{s-t^{1}_{s,N}-1}g_{s_{1},N_{1}}(x)
+Q_{s-t^{1}_{s,N}-2}g_{s_{1},N_{1}}(x)\Big)\in L^{1}
\end{equation}
and
\begin{equation}\label{eq:5.11}
Q_{s-t^{1}_{s,N}-i}g_{s_{1},N_{1}}(x)\in H^{1}(\mathbb{R}^{n}), i=0,1,2.
\end{equation}
Similar to case 1, by (\ref{5.0}),
the fact that
$$T^{2}_{s, t^{1}_{s,N},N}R_{i}g_{s_{1},N_{1}}(x)\in L^{1}(\mathbb{R}^{n}),
i=1,\cdots,n,$$ implies that for $ i=1,\cdots,n,$
$$R_{i}\Big[g^{3}_{s,N}(x)-\Big(Q_{s-t^{1}_{s,N}-1}g_{s_{1},N_{1}}(x)
+Q_{s-t^{1}_{s,N}-2}g_{s_{1},N_{1}}(x)\Big)\Big]\in
L^{1}(\mathbb{R}^{n}).$$
 Therefore we  have
$$g^{3}_{s,N}(x)-\Big(Q_{s-t^{1}_{s,N}-1}g_{s_{1},N_{1}}(x)
+Q_{s-t^{1}_{s,N}-2}g_{s_{1},N_{1}}(x)\Big)\in H^{1}(\mathbb{R}^{n}).$$ Hence
$g^{3}_{s,N}(x)\in H^{1}(\mathbb{R}^{n})$.

For $i=1,\cdots,n$, we have
$$T^{1}_{s, t^{1}_{s,N},N}R_{i}g_{s_{1},N_{1}}(x)=
T^{1}_{s,
t^{1}_{s,N},N}R_{i}\Big[g^{1}_{s,N}(x)+g^{2}_{s,N}(x)+Q_{s-t^{1}_{s,N}-1}g_{s_{1},N_{1}}(x)\Big].$$
So the conditions
$$T^{1}_{s,
t^{1}_{s,N},N}R_{i}\Big[g^{1}_{s,N}(x)+g^{2}_{s,N}(x)+Q_{s-t^{1}_{s,N}-1}g_{s_{1},N_{1}}(x)\Big]
\in P^{\alpha}(\mathbb{R}^{n}), i=1,\cdots,n $$
 and $g^{1}_{s,N}(x)\in P^{\alpha}(\mathbb{R}^{n})$ implies
$$T^{1}_{s, t^{1}_{s,N},N}R_{i}\Big[g^{2}_{s,N}(x)+Q_{s-t^{1}_{s,N}-1}g_{s_{1},N_{1}}(x)\Big]
\in P^{\alpha}(\mathbb{R}^{n}).$$

For $i=1,\cdots,n$, we have
\begin{eqnarray*}
&&T^{1}_{s, t^{1}_{s,N},N}R_{i}\Big[g^{2}_{s,N}(x)+Q_{s-t^{1}_{s,N}-1}g_{s_{1},N_{1}}(x)\Big]\\
&=&R_{i}\Big[g^{2}_{s,N}(x)-Q_{s-t^{1}_{s,N}-2}g_{s_{1},N_{1}}(x)\Big]\\
&&+T^{1}_{s,
t^{1}_{s,N},N}R_{i}\Big[Q_{s-t^{1}_{s,N}-2}g_{s_{1},N_{1}}(x)
+ Q_{s-t^{1}_{s,N}-1}g_{s_{1},N_{1}}(x)\Big].
\end{eqnarray*}
 Applying (\ref{eq:5.11}), we obtain
$$R_{i}\Big[g^{2}_{s,N}(x)-Q_{s-t^{1}_{s,N}-2}g_{s_{1},N_{1}}(x)\Big]\in
P^{\alpha}(\mathbb{R}^{n}).$$ Hence
$$g^{2}_{s,N}(x)-Q_{s-t^{1}_{s,N}-2}g_{s_{1},N_{1}}(x)\in P^{\alpha}(\mathbb{R}^{n}).$$
That is to say, $g_{s_{1},N_{1}}(x)$ satisfies the conditions
(\ref{eq:5.10}) and (\ref{eq:5.11}). By (\ref{eq:5.5}) and (\ref{eq:5.6}),
$$g^{1}_{s,N}(x)\in P^{\alpha}(\mathbb{R}^{n})$$
and
$$g^{3}_{s,N}(x)\in H^{1}(\mathbb{R}^{n}).$$
 Putting  together, we complete the proof of
(\ref{eq:5.3}).

\end{document}